\documentclass[10pt,reqno,final]{amsart}
\usepackage{epsfig,amssymb,amsmath,mathtools,version}
\usepackage{amssymb,version,graphicx,fancybox,mathrsfs}
\usepackage{hyperref}
\hypersetup{colorlinks=true,citecolor=blue,linkcolor=red,anchorcolor=blue}
\usepackage[notcite,notref]{showkeys}
\usepackage{subfigure}
\usepackage{color}
\usepackage{stmaryrd}
\usepackage{multirow}
\usepackage{booktabs,siunitx}
\usepackage{multicol}
\usepackage{enumitem}
\usepackage{algorithm,setspace}
\usepackage{algpseudocode}
\algrenewcommand\algorithmicrequire{\textbf{Inputs:}}
\algrenewcommand\algorithmicensure{\textbf{Outputs:}}
\usepackage{epstopdf}
 \usepackage[foot]{amsaddr}
 
\usepackage[marginparwidth=3cm, marginparsep=3mm]{geometry}

\usepackage{textgreek}
\usepackage{upgreek}

\usepackage{relsize}

\DeclareTextAccent{\myacc}{T1}{4}

\usepackage{marginnote}

\catcode`\@=11 \theoremstyle{plain}
\@addtoreset{equation}{section}   

\@addtoreset{figure}{section}
\renewcommand\thefigure{\thesection.\@arabic\c@figure}
\newtheorem{thm}{\bf Theorem}[section]

\newenvironment{theorem}{\begin{thm}} {\end{thm}}
\newtheorem{cor}{\bf Corollary}

\newenvironment{corollary}{\begin{cor}} {\end{cor}}
\newtheorem{lmm}{\bf Lemma}[section]

\newenvironment{lemma}{\begin{lmm}}{\end{lmm}}
\theoremstyle{remark}
\newtheorem{rem}{\bf Remark}[section]
\theoremstyle{definition}
\newtheorem{definition}{\bf Definition}[section]

\definecolor{ligreen}{rgb}{0.0, 0.3, 0.0}

\definecolor{darkblue}{rgb}{0.0, 0.0, 0.55}

\definecolor{anti-flashwhite}{rgb}{0.55, 0.57, 0.68}

\definecolor{teal}{RGB}{0,128,128}

\newcommand{\bs}[1]{\boldsymbol{#1}}
\allowdisplaybreaks

\definecolor{darkblue}{rgb}{0,0,0.75}
\newcommand{\wll}[1]{#1}

\begin{document}
\bibliographystyle{plain}
\baselineskip 13pt

\title[Approximate Kolmogorov Superpositions] {Explicit  Construction of Approximate Kolmogorov Superpositions with $C^2$-Smoothness}

\author[Approximate Kolmogorov  Superpositions]{Lunji Song$^\dagger$}
\address{${}^\dagger$\rm School of Mathematics and Statistics, Lanzhou University, Lanzhou 730000, China. The research of the first author is partially supported by the National Natural Science Foundation of China (Grant No. 12171216). Email: song@lzu.edu.cn.}

\author{Zilan Cheng$^\ddagger$}
\address{$^\ddagger$\rm Division of Mathematical Sciences, School of Physical and Mathematical Sciences, Nanyang Technological University, 637371, Singapore. Email: zilan001@e.ntu.edu.sg.}
\author{Juan Diego Toscano$^\S$}
\address{${}^\S$\rm Division of Applied Mathematics, Brown University, Providence, 02912, RI, USA.}

\author{Li-Lian Wang$^\ddagger$}
\address{$^\ddagger$\rm Division of Mathematical Sciences, School of Physical and Mathematical Sciences, Nanyang Technological University, 637371, Singapore. The research of the second and the fourth authors is partially supported by  Singapore MOE AcRF Tier 2 Grant: MOE-T2EP20224-0012. Email: lilian@ntu.edu.sg.}





\begin{abstract} We explicitly construct an approximate version of the \wll{Kolmogorov superpositions}, which is composed of $C^2$-inner and outer functions, and  can approximate an  arbitrary $\alpha$-H\"older continuous function  $f\in {\mathcal H}^{\alpha}([0,1]^d),\, \alpha\in (0,1]$ with accuracy \wll{$O(N^{-\alpha})$,}  where 
$N$ denotes the number of outer summations.   The inner functions are generated by applying suitable translations and dilations to a piecewise $C^2,$ strictly-increasing function, while the outer functions are constructed row-wise through piecewise $C^2$-interpolation using newly designed shape functions. This novel variant of Kolmogorov superpositions overcomes the ``wild'' and ``pathological'' behaviors of the inherent single variable functions, but retains the essence of Kolmogorov's strategy of exact representation---an objective that Sprecher (Neural Netw.\! 144 (2021)\! 438–442) has actively pursued.  We also discuss the implications of this new construction and demonstrate its applicability to related neural networks.
\end{abstract}

\maketitle

\vspace*{-12pt}

\section{Introduction}
\setcounter{equation}{0}
\setcounter{lmm}{0}
\setcounter{thm}{0}

 In a series of seminal papers~\cite{Kolmogorov1956,Arnold1957a,Arnold1957b,Kolmogorov1957}, Kolmogorov and Arnold studied how multivariate continuous functions defined on a bounded domain can be represented as superpositions of continuous functions of fewer variables.
 Their significant discovery, known as the \wll{Kolmogorov superposition theorem (KST)}, was stated in Kolmogorov~\cite{Kolmogorov1957} (1957):
 {\em For any dimension $d \geq 2,$ there are continuous real functions $\psi_{p, q}(x)$ on the closed unit interval $E^1=[0, 1]$ such that every continuous real function $f(x_1, \ldots, x_d)$ on the $d$-dimensional unit cube $E^d=[0,1]^d$ is representable as}
\begin{equation}\label{KART-0}
f(x_1, \ldots, x_d)=\sum_{q=0}^{2 d} g_q\circ \sum_{p=1}^d \psi_{p, q}(x_p),
\end{equation}
{\em where $g_q$ are continuous real functions on $\mathbb R=(-\infty,\infty)$.} The univariate functions  $\psi_{p,q}$ are called the {\em inner functions} (which are  independent of $f$), while $g_q$ are referred to as the {\em outer functions} (which depend on $f$).  This mathematically elegant representation can be regarded as a refutation of Hilbert's 13th problem: {\em ``There exist continuous functions of three variables that cannot be represented as superpositions of continuous functions of two variables''} (see~\cite{Girosi1989}). 

Although the proof in \cite{Kolmogorov1957} was brief and not highly constructive, the theory \wll{itself}  has nonetheless inspired extensive research in at least two interwoven directions:  (i) construction of refined representations/versions 
(see e.g.,  \cite{Lorentz1962metric,Sprecher1963dissertation,Sprecher1965structure,Lorentz1966approximation,Kahane1975,Lorentz1996,Sprecher1996numerical,Sprecher1997numerical,Koppen2002,braun2009application}, Sprecher's book~\cite{Sprecher2017algebra} and recent works 
\cite{schmidt2021kolmogorov,Laczkovich2021superposition,ismayilova2024kolmogorov,Ismailov2024kart_uar_nn}), 
and  (ii) applications in network architectures and computations (see e.g., \cite{HechtNielsen1987,Girosi1989,Kourkova1991,Kuurkova1992kolmogorov,Igelnik2003,Guliyev2018approximation,Montanelli2020error,Shen2021three,He2024,Lai2024optimal,LaiShen2025Optimal,Liu2024KAN,Shukla2024comprehensive,toscano2025aivt,toscano2026,Toscano2024kkans,Guilhoto2025deep,toscano2026g} and references therein).  Lorentz~\cite{Lorentz1962metric} was the first to observe that the outer functions \( g_q \) can be chosen identically. Subsequently, Sprecher~\cite{Sprecher1963dissertation} demonstrated that the inner functions \( \psi_{p,q}(x) \) may be taken as \( \lambda_p \psi_q(x) \) for suitable constants $\lambda_p.$ \wll{The version due to Lorentz \cite{Lorentz1966approximation} and Kahane
\cite{Kahane1975}, as stated in \cite[Chapter~17]{Lorentz1996}, takes
the univariate functions in \eqref{KART-0} to be \(g_q = g\) and
\(\psi_{p,q} = \lambda_p \psi_q\). Kahane's proof was based on the
Baire category theorem.}
Sprecher~\cite{Sprecher1963dissertation,Sprecher1996numerical,Sprecher1997numerical} provided a more  computer-friendly construction of a different version presented in~\cite{Sprecher1997numerical}: {\em Let  $m\ge 2d, \gamma\ge m+2$ be given integers, and then any $f\in C([0,1]^d)$ can be represented by the superpositions:
\begin{equation}\label{Sprecher-0}
f(x_1, \ldots, x_d)=\displaystyle\sum_{q=0}^{m} {g_q}\circ\sum_{p=1}^d \lambda_p {\psi}(x_p+q a),
\end{equation}
where $\psi$ is  of 
$\frac{\ln 2} {\ln \gamma}$-Lipschitz class and  increasing, $g_q$ are continuous, and the parameters}  
\begin{equation}\label{alambda} 
a=\frac{1}{\gamma(\gamma-1)};\quad   \lambda_1=1, \quad  \lambda_p=\sum_{r=1}^{\infty} \gamma^{-(p-1)(d^r-1) /(d-1)},\;\;\; 2\le p\le d.
\end{equation} 
It is noteworthy that the function $\psi$ in \eqref{Sprecher-0} is defined via a limiting process:
\(\psi(x) = \lim\limits_{k \to \infty} \psi_k(d_k),\)
where \( d_k\in [0,1] \) are pre-assigned terminating rational numbers expressed in terms of \(\gamma\)-based decimals, and the values \(\psi_k(d_k)\) are computed using Köppen's recursive formula~\cite{Koppen2002} (modified from the formula in \cite{Sprecher1996numerical} so that the resulting function 
\(\psi\) is increasing, as rigorously shown in~\cite{Braun2009constructive}).

It is common that in most, if not all, existing constructions, the inner functions consist of staircase-like structures with derivatives vanishing almost everywhere, exhibiting “wild” and “pathological” behaviors (see  \cite{Girosi1989,Sprecher2017algebra,Demb2021note}). 
Indeed, the  exact superposition  fails to be true  if the inner functions are required to be smooth, owing to  the assertion by 
Vitushkin \cite{Vitushkin1954,Vitushkin2004} (on Hilbert's 13th problem): {\em there are $r\ge 1$ times continuously differentiable functions \wll{of $d$ variables}  not representable by $r$ times continuously differentiable functions of less than $d$ variables.}  Accordingly, any attempt to construct smooth variants of exact superpositions is futile, which inherently obstructs the direct application of the theory to neural networks \cite{HechtNielsen1987,Girosi1989,Demb2021note}.

K{$\dot {\rm u}$}rkov{\'a}~\cite{Kourkova1991,Kuurkova1992kolmogorov} suggested that one should sacrifice the exactness of representation by adopting an approximate version instead.   
In a recent note, Demb and Sprecher \cite{Demb2021note} raised the   question \wll{on constructing approximate Kolmogorov's superpositions}: 
 {\em Are there means of smoothing these mappings {\rm(}i.e., $\Psi_q=\sum_{p=1}^d \lambda_p \psi_q(x_p): [0,1]^d\mapsto [0,1]$  in any approximate version{\rm)} into more computer-friendly forms while retaining the essence of Kolmogorov's strategy of exact representation?} Although this open question was not answered in \cite{Demb2021note}, they managed to construct the superpositions more suitable for parallel computations, which could approximate any continuous function $f$ on 
$[0,1]^d$ to an error $\varepsilon\|f\|_\infty$ for any preassigned $\varepsilon>0:$
\begin{equation}\label{Demb-Sprecher21}
 \Big\|f(x_1,\ldots, x_d)-
 \sum_{q=0}^m g_q\circ \sum_{p=1}^d \lambda_p \psi(x_p+q a)\Big\|_\infty < \varepsilon \|f\|_\infty,   
\end{equation}
where $m$ is an integer depending on $\varepsilon,$ and $\lambda_p$ and $a$ are constants. Nevertheless,  the inner functions and Kolmogorov maps $\Psi_q$ therein still suffer from ``wildness'' (see \cite[Figure~2]{Demb2021note}). 
Igelnik and Parikh~\cite{Igelnik2003} attempted much earlier for approximate superpositions involving \( C^4 \)-smooth inner and outer functions, and claimed that the resulting version   could approximate any continuously differentiable function \( f \) with a bounded gradient on \( [0,1]^d \), with convergence guarantees:
\begin{equation}\label{IgeP03}
 \Big\|f(x_1,\ldots, x_d)-
 \sum_{q=1}^N g_q\circ \sum_{p=1}^d \lambda_p \psi_q(x_p)\Big\|_\infty =O\Big(\frac 1 N\Big),   
\end{equation}
 where $\sum_{p=1}^d \lambda_p \leq 1, \lambda_p>0$ are given rationally independent numbers (that is, the equation \( \sum_{p=1}^d r_p \lambda_p = 0 \), with rational coefficients \( r_p \), implies that all \( r_p = 0 \)). 
 Their construction essentially follows the constructive proof of Lorentz's version in Kahane \cite{Kahane1975} (also see \cite{Morris2021Hilbert13}).
The inner functions in \cite{Igelnik2003} are composed of constant segments over closed sub-intervals of length $O(1/N)$, interspersed with cubic spline transitions across gaps of length  $O(1/N^2)$ (which turned out to be the \wll{main} difference from Kahane \cite{Kahane1975},  where linear pieces were used to fill in the gaps), while their outer functions are constructed via piecewise linear and cubic spline interpolation of samples of \( f \) on hypercubes formed by disjoint sub-intervals. Igelnik and Parikh~\cite{Igelnik2003} further asserted that there exists a positive constant $C$ independent of $N$ and gap size $\delta$ such that  
\begin{equation}\label{c4bound}
    \|\psi_q^{(k)}\|_\infty,\; \|g_q^{(k)}\|_\infty \le C,\quad k=1,2,3,4, \;\; q=1,\ldots, N,
\end{equation}
which is essential for configuring the Kolmogorov spline network (KSN) with a parameter count of \( P = O(N^{3/2})\), \wll{independent of $d$.}   
However, the magnitudes of the derivatives grow rapidly—often explosively—with powers of \( N \) in the small gaps, so \eqref{c4bound} cannot hold. Moreover, the derivatives of inner functions vanish on the disjoint closed sub-intervals, which are undesirable \cite{Demb2021note}.  

The aim of this paper is to provide a \wll{positive} answer to the aforementioned open question posed in \cite{Demb2021note} by
 explicitly constructing  novel  $C^2$-smooth approximate  Kolmogorov superpositions: 
\begin{equation}\label{fx-KA-form00}
    f(\bs x)\approx  f_N(\bs x)=\sum_{q=1}^{N} g_q\circ \sum_{p=1}^d \lambda_p \psi_q(x_p)= \sum_{q=1}^{N} g_q\circ \Psi_q(\bs x),
\end{equation}
\noindent with the following important features: 
\begin{itemize}
\item[(i)] The inner functions $\psi_q\in C^2([0,1])$ are strictly increasing and can be generated from a single function $\phi(x)$ through suitable translations and dilations 
(see \eqref{PsiqxAS}). \wll{This piecewise $C^2$ function $\phi(x)$ is constructed by gluing new shape functions, consisting of different pieces on gaps of size $\delta=N^{-2}$ and on subintervals of size $h-\delta$ with $h=N^{-1}$ (see \eqref{shapfunc0}). 
 They have controllable derivative values in the small gaps and do not vanish in the disjoint closed sub-intervals (see Lemma \ref{Shape-Lemma0} and \eqref{innerA00}).}
\smallskip
\item[(ii)] Under Kahane’s configuration \cite{Kahane1975} of multilevel partitions of $[0,1]^d$ into hypercubes separated by gaps,
we determine a set of $\bs \lambda^*=(\lambda_1^*,\ldots, \lambda_d^*)$ with 
$|\bs \lambda^*|_1=\sum_{p=1}^d \lambda_p^*=1$  and show that the Kolmogorov maps  
\begin{equation}\label{Psiqq}
\Psi_q(\bs x)=\lambda_1^*\,\psi_q(x_1)+\cdots+ \lambda_d^*\,\psi_q(x_d):\, [0,1]^d\mapsto [0,1],\;\; 1\le q\le N,  
\end{equation}
can well separate the hypercubes and allow for a global interpolation to construct the outer functions; see Theorem \ref{Hypercube-Sep} and  
Definition \ref{Defn:Outer01}.
Interestingly, the scaling $\lambda_p^*=O(N^{p-d})$ 
(see Corollary \ref{lamda-small})  is reminiscent of the parameter choices in Sprecher’s construction \cite{Demb2021note}, which are derived from K\"oppen’s recursive relations \cite{Koppen2002}.  

\smallskip
\item[(iii)]  We rigorously show that the resulting approximate superpositions in \eqref{fx-KA-form00}
with $\lambda_p=\lambda_p^*$ 
can approximate any $\alpha$-H\"older continuous $f$ with $\alpha\in (0,1],$ within an accuracy $O(N^{-\alpha}).$ See the main result stated in Theorem 
\ref{MainResult}. \wll{Moreover, our explicit construction allows us to verify the convergence rate directly, which appears to be difficult, or has not been carried out, for most existing constructions.}
\smallskip 
\item[(iv)] \wll{Although the geometric setup follows \cite{Kahane1975,Igelnik2003}, our explicit construction of the inner functions enables a delicate analysis of the separation and ordering of the hypercubes, leading to the choice of parameters that can mitigate dislocations and nonlocal memory references of the singular Kolmogorov maps  \cite{Sprecher2017algebra,Demb2021note}. Such separations have not been explored in Igelnik and Parikh~\cite{Igelnik2003}, where the parameters $\lambda_p$ are assumed to be rationally independent as in \cite{Kahane1975}, and appear to be subtle in the discrete setting considered in \cite{Sprecher2017algebra,Demb2021note}.}
\end{itemize}
\smallskip

This explicit construction carries significant implications. Firstly, it inherently suggests an algorithm whose high parallelism and intrinsic properties — such as the strict monotonicity and smoothness of the inner functions — can, and perhaps should, be preserved in any implementation of Kolmogorov-based neural networks and applications (see, e.g., \cite{Leni2014Progressive,Montanelli2020error,He2024}).   
Secondly, it can act as an intermediate framework in which the one-dimensional functions are further approximated or parameterized by various ansatz classes (e.g., splines, radial basis, spectral basis, and multilayer perceptrons), with smoothness playing a critical role in establishing convergence and providing estimates for the required network size. After all, the non-smooth but continuous exact representations can only guarantee the universal approximability of Kolmogorov networks using various ansatz expansions/interpolation/parameterisations (see e.g., \cite{Kourkova1991,Kuurkova1992kolmogorov,Toscano2024kkans,PetersenZech2024}).  It is still open to show its convergence rate.           

The rest of the paper is organized as follows. In Section \ref{Sect2}, we introduce the new shape functions and construct the inner functions.  In Section \ref{Sub:center}, we analyze the separation properties of the Kolmogorov maps and explicitly construct the outer functions and approximate  
Kolmogorov superpositions. We prove  
 convergence of the approximate version and provide some numerical verifications of the expected convergence rate in  Section   
\ref{Sect:MainProof-1}.  We then conclude the paper with some discussions and remarks.

\section{Shape functions and construction of inner functions}\label{Sect2}
\setcounter{equation}{0}
\setcounter{lmm}{0}
\setcounter{thm}{0}

In this section, we construct the inner functions. We begin by introducing the shape functions, which constitute the fundamental building blocks for the construction.  These functions are then assembled to form a strictly increasing, piecewise $C^2$ function $\phi$, which generates the inner functions $\psi_q$ through appropriate translations and dilations. 
 With this in place, we define the Kolmogorov maps \(\Psi_q(\bs{x}; \bs{\lambda})\)
 with respect to  \(\bs{\lambda}\) and the parameters $\nu, h.$
 

\smallskip

We first introduce two types of shape functions.
Consider the smeared-out $C^2$-Heaviside function (see e.g., \cite[(1.22)]{Osher2003}):  \begin{equation*}\label{HeavideA}
    H_\epsilon(z)=
\begin{dcases}0, & z< -\epsilon \\
\frac{1}{2}+\frac{z}{2\epsilon}+\frac{1}{2\pi}\sin\frac{\pi z}{\epsilon},\quad  &-\epsilon \le z \le \epsilon,\\
1, & z> \epsilon.
\end{dcases}
\end{equation*}
We  extract  the smeared-out portion of $H_\epsilon(z), z\in [-\epsilon,\epsilon],$ and transform it to the reference interval $[0,1]:$
\begin{equation}\label{St}
    S(z):=z-\frac{1}{2\pi} \sin{2\pi z},\quad z\in [0,1].
\end{equation}
One can readily verify that
\begin{equation}\label{sproperty}
     S(1)=1,\;\; S(0)=S'(0)=S''(0)=S'(1)=S''(1)=0,\;\; S(1/2)=1/2,
\end{equation}
and $S(z)$ is strictly  increasing in $(0,1)$. 

Another piece of the puzzle is the \wll{quintic} polynomial:  
\begin{equation}\label{P-fun} 
P(z):=\int_0^{z}\tau^2(1-\tau)^2 d\tau\Big/\int_0^{1}\tau^2(1-\tau)^2 d\tau= z^3(6z^2-15z+10), 
\end{equation}
which also satisfies 
\begin{equation}\label{Pproperty}
     P(1)=1,\;\; P(0)=P'(0)=P''(0)=P'(1)=P''(1)=0, \;\; P(1/2)=1/2,
\end{equation}
 and is strictly  increasing as $P'(z)=30z^2(1-z)^2$ in $(0,1).$

We next glue $S$ and $P$ piecewise to construct a one-dimensional function generating the inner functions and the associated Kolmogorov maps.
  To do this, we first lay out the partition and grids as in \wll{Kahane's constructive proof 
  \cite{Kahane1975} (also see \cite{Lorentz1996,Igelnik2003}).} 
  \wll{Let $\mathbb{N}$ and $\mathbb{R}$ denote the sets of natural numbers and real numbers, respectively.
For  $N\in \mathbb N,$ set 
$\delta=\tfrac 1 {N^2}$ and 
$h=N\delta=\tfrac 1 N.$ Further, let $\nu\in (0,h)$ be 
a tuning parameter to be specified later.} 
We generate two sets of interlacing grids on the real line:
\begin{equation}\label{tjhd}
t_j=jh, \quad t_j^\delta=t_j+\delta=jh+\delta,\;\;\; j=0,\pm 1,\ldots,
\end{equation}
which naturally induce  three types of closed sub-intervals:
\begin{equation}\label{1D-P0A} 
  B^j:=[t_j,t_{j+1}], \quad  G^j:=[t_j, t_j^\delta],\quad   I^j:=[t_j^\delta, t_{j+1}],
\end{equation}
referred to as the $j$-th {\em block}, {\em gap}, and {\em interval}, respectively. Clearly, $B^j=G^j\cup I^j,$ with the intervals $I^j$ of length $(N-1)\delta$  separated by small gaps $G^j$  of length $\delta$ (see Figure \ref{figsubLq2D} (a)). 
We prescribe the ``interpolating'' data  $\{(t_j, y_j),\; (t_j^\delta, y_j^\nu)\}$:  
\begin{equation}\label{yjdata}
    y_j=t_j=jh,\quad  y_j^\nu=t_j+\nu=jh+\nu,
\end{equation}
and on each block, we define the following piecewise function:~for  $j=0,\pm 1,\ldots,$ 
\begin{equation}\label{shapfunc0}
\begin{split}
\phi(t)\big|_{t\in B^j}&=
\begin{dcases}
y_j+(y_j^\nu-y_j)\, S\big(\tfrac{t-t_j}{t_j^\delta-t_j}\big), &t\in G^j,\\
y_j^\nu+(y_{j+1}-y_j^\nu)\,P\big(\tfrac{t-t_j^\delta}{t_{j+1}-t_j^\delta}\big), &t\in I^j,
\end{dcases}\\[4pt]
&=
\begin{dcases}
jh+\nu\, S\big(\tfrac{t-jh}\delta\big), &t\in [jh, jh+\delta],\\
jh+\nu+(h-\nu)\,P\big(\tfrac{t-jh-\delta}{h-\delta}\big), &t\in [jh+\delta, (j+1)h].
\end{dcases}
\end{split}
\end{equation}

\wll{We observe that, on each block $B^j$, \(\phi\) consists of a \(P\)-piece on the
subinterval of length \(h-\delta\) and an \(S\)-piece on the gap of length
\(\delta\). Moreover, \(\phi\) is strictly increasing and maps
\(B^j=[jh,(j+1)h]\) onto itself. We highlight that the parameter $\nu$ controls the jumps at the junctions between the gaps and the hypercubes, and tunes the deviations from the constant values on the subintervals $I^j$. Indeed, we find this newly introduced parameter plays an important role in our construction, 
whereas in, for example, \cite{Kahane1975,Lorentz1996,Igelnik2003}, constant pieces were used in the subintervals.
We summarize below the properties of the generating function $\phi$.} 
\begin{lemma}\label{Shape-Lemma0}  
Let $\phi(t)$ be the piecewise function defined in \eqref{shapfunc0}. Then we have the following properties.
\begin{itemize}
    \item[{\rm (i)}] $\phi(t)\in C^2$ is strictly  increasing,      and satisfies
\begin{equation}\label{intercond}
\begin{split}
\phi(t_j)=t_j,\quad \phi(t^{\delta}_j)=t_j+\nu,\quad  \phi^{\prime}(t_j)=\phi^{\prime\prime}(t_j)=\phi^{\prime}(t^{\delta}_j)=\phi^{\prime\prime}(t^{\delta}_j)=0.
\end{split}
\end{equation}

\item[{\rm (ii)}] 
For the pieces on  $ G^j=[jh, jh+\delta],$
\begin{equation}\label{bounds}
   \max_{t\in G^j}\phi(t)\le jh+\nu,\quad    
   \max_{t\in G^j} |\phi'(t)|\le 2\,\frac{\nu}{\delta},\quad   \max_{t\in G^j}|\phi^{\prime\prime}(t)|\leq 2\pi\,\frac {\nu}{\delta^{2}},
\end{equation}
and on $ I^j=[jh+\delta,(j+1)h],$
\begin{equation}\label{bounds2}
   \max_{t\in I^j} \phi(t)\le (j+1)h,\quad 
   \max_{t\in I^j}|\phi'(t)|\le \frac{15}{8}\,\frac{
   h-\nu}{h-\delta},\quad   \max_{t\in I^j}|\phi^{\prime\prime}(t)|\leq  \frac{10}{\sqrt{3}}\,\frac{h-\nu}{(h-\delta)^2}.
\end{equation}
\end{itemize}
\end{lemma}
\begin{proof} The property  (i) follows from \eqref{St}-\eqref{P-fun} and  \eqref{shapfunc0} straightforwardly.   As  $S'(z), P'(z)>0$ on $(0,1),$ the strict monotonicity of $\phi(t)$ is obvious. 
(ii) Note that   
$$ 
S(z)\in[0,1],\quad S'(z)\in [0, 2], \quad  S''(z)\in [-2\pi, 2\pi],  
$$
so we can derive the bounds in \eqref{bounds}  from \eqref{shapfunc0} and direct calculation. 
Similarly, we can obtain \eqref{bounds2}.
\end{proof}
\begin{rem}\label{Rmk:PS} \wll{\em  It is clear from \eqref{shapfunc0} that two different shape functions are used to construct $\phi(t)$. This is mainly because, on the subintervals, the polynomial $P$ maps rational numbers to rational numbers, which ensures the separation property for rationally independent $\lambda_p$, as in \cite{Kahane1975,Igelnik2003}. In contrast, the function $S$ is used in the gaps since its derivative has a relatively smaller magnitude. In principle, $S$ could be replaced by $P$ in the gaps, but \eqref{shapfunc0} on a block would still be defined in two pieces.   \qed    
}
\end{rem}

\wll{The construction will only involve $(N+1)$ blocks with $j\in {\mathcal J}:=\{-1,\ldots, N-1\},$ so we restrict $\phi(t)$ on $[-h,1].$}
To construct the inner functions,   we generate an $N$-level (or $N$-story)  covering of  $[0,1]$ by $\delta$-shifting  of the closed intervals and gaps defined in \eqref{tjhd}-\eqref{1D-P0A} (see  \cite{Kahane1975} or Figure \ref{figsubLq2D} (a) for an illustration of $N=5$ and $\delta=1/25$):
\begin{equation}\label{1D-P0qq}
 B_q^j:=B^j+(q-1)\delta= 
 [jh+(q-1)\delta,(j+1)h+(q-1)\delta],
\end{equation}
and likewise   
\begin{equation}\label{1D-G0qq} 
 G_q^j:=G^j+(q-1)\delta,\quad  I_q^j:=I^j+(q-1)\delta,
\end{equation}
for $q=1,\ldots, N.$
This leads to the multi-level  covering of $[0,1]$:
\begin{equation}\label{LqdefnL}
L_q:=\Big(\bigcup_{j=-1}^{N-1}  I_q^j\Big)\cap [0,1].
\end{equation} 

\wll{We construct  the inner functions through  linear transformations and suitable normalization of $\phi(t)$ defined in
\eqref{shapfunc0}.} 
\begin{definition}[{\bf Inner functions}]\label{Inner-fun}  \emph{Let $\phi(t)$ be the piecewise $C^2,$ strictly increasing function defined in
\eqref{shapfunc0}, and introduce the linear transformations: \begin{equation}\label{shift-coord}
x=t+(q-1)\delta\;\;\; {\rm or} \;\;\; t=x+(1-q)\delta, \quad  t\in B^j,\;\; x\in B_q^j.
\end{equation} 
We define the inner function  at the $q$-th level by
\begin{equation}\label{PsiqxAS}
\begin{split}
\psi_q(x)& =\frac{\phi(x+(1-q)\delta)- \phi((1-q)\delta)} {\phi(1+(1-q)\delta)-\phi((1-q)\delta)}\\[4pt]
&=\phi(x+(1-q)\delta)- \phi((1-q)\delta), \quad  x\in B_q^j\cap [0,1],
\end{split}
\end{equation}
 for $j\in \mathcal J$ and $q=1,\ldots, N.$} 
\end{definition}

\wll{Since 
$$(1-q)\delta\in I^{-1}=[\delta-h,0],\quad 1+(1-q)\delta\in I^{N-1}=[1-(N-1)\delta,1],$$  we find from \eqref{shapfunc0} with $j=-1,N$ that 
\begin{equation}\label{phi-delta}
\begin{split}
  \phi((1-q)\delta)=(h-\nu)\big(P(\tfrac{N-q}{N-1})-1\big),\quad  \phi(1+(1-q)\delta)=1+(h-\nu)\big(P(\tfrac{N-q}{N-1})-1\big).
  \end{split}
  \end{equation}
Thus, 
\begin{equation}\label{phi-delta-1}
\begin{split}
  \phi(1+(1-q)\delta)-\phi((1-q)\delta)=1,
  \end{split}
  \end{equation}
   leading to the second identity in \eqref{PsiqxAS}.} 
In view of $\psi_q(0)=0,\,\psi_q(1)=1,$ and   $\phi(t)$ \wll{being}  strictly increasing, $\psi_q(x)$ is also strictly increasing and 
its range is  $[0,1].$ 
Moreover,  we find from  \eqref{phi-delta} that for $1\le q\le N,$
  \begin{equation}\label{phiqibnd}
     \nu-h\le \phi((1-N)\delta) \le \phi((1-q)\delta)\le \phi(0)=0.
  \end{equation}

As a direct consequence of Lemma \ref{Shape-Lemma0}, Definition \ref{Inner-fun}   and the above,  the inner functions
$\psi_q$ enjoys the following properties. 
For each $1\le q\le N,$ the inner function $\psi_q(x)\in C^2([0,1])$ is strictly increasing 
and its range is $[0,1].$ Moreover, 
if $h> \nu>0,$ then 
\begin{equation}\label{innerA00}
\begin{split}
& 0\le \psi_q'(x)\le \max\Big\{ 
\frac {2\nu}{\delta}, \frac{15}{8}\,\frac{
   h-\nu}{h-\delta}\Big\},\quad  |\psi_q''(x)|\le \max\Big\{ 
\frac {2\pi\nu}{\delta^{2}}, \frac{10}{\sqrt{3}}\,\frac{h-\nu}{(h-\delta)^2} \Big\}.
\end{split}
\end{equation}

In Figure \ref{figsubLq2D}, we illustrate the partitions of $[0,1]$ at different levels by shifting, and plot the corresponding inner functions $\psi_q(x)$ for \wll{$N=5,\delta=1/25,$ and $\nu=\delta^2,$} which demonstrate their smoothness and strict monotonicity. 
\begin{figure}[!h]
\centering
 \subfigure[Intervals+Gaps at $L_q$]{
    \includegraphics[width=0.24\textwidth]{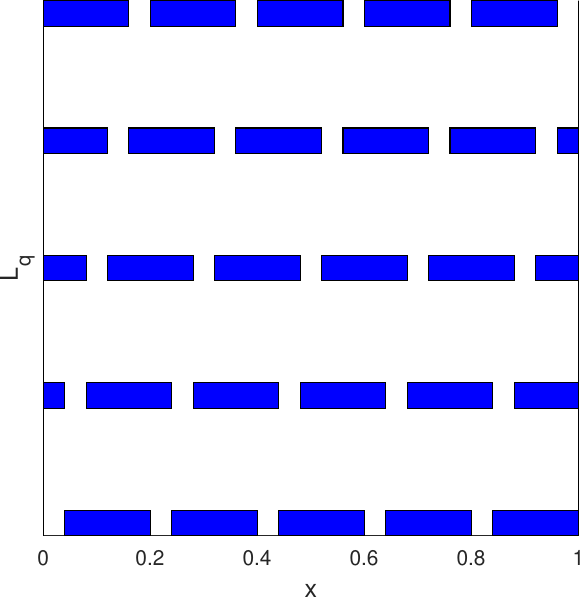}
     }\quad 
      \subfigure[$\psi_1(x)$ at $L_1$]{
    \includegraphics[width=0.25\textwidth]{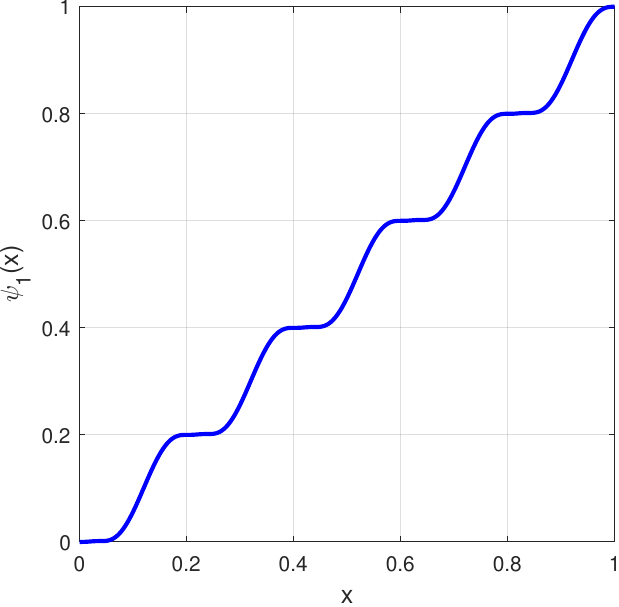}
     } \quad 
        \subfigure[$\psi_2(x)$ at $L_2$]{
    \includegraphics[width=0.25\textwidth]{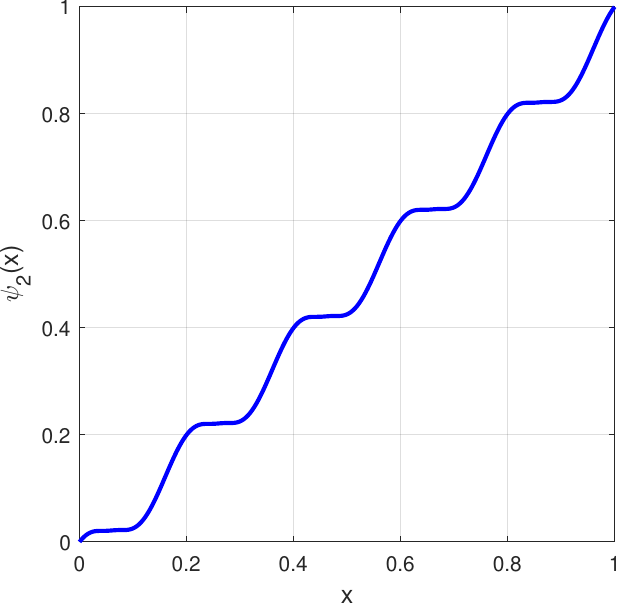}
     }
     
        \subfigure[$\psi_3(x)$ at $L_3$]{
    \includegraphics[width=0.25\textwidth]{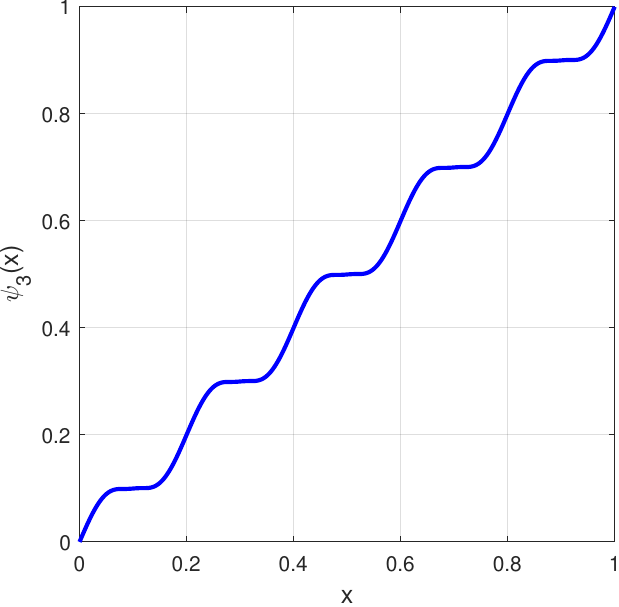}
     }
     \quad 
     \subfigure[$\psi_4(x)$ at $L_4$]{
    \includegraphics[width=0.25\textwidth]{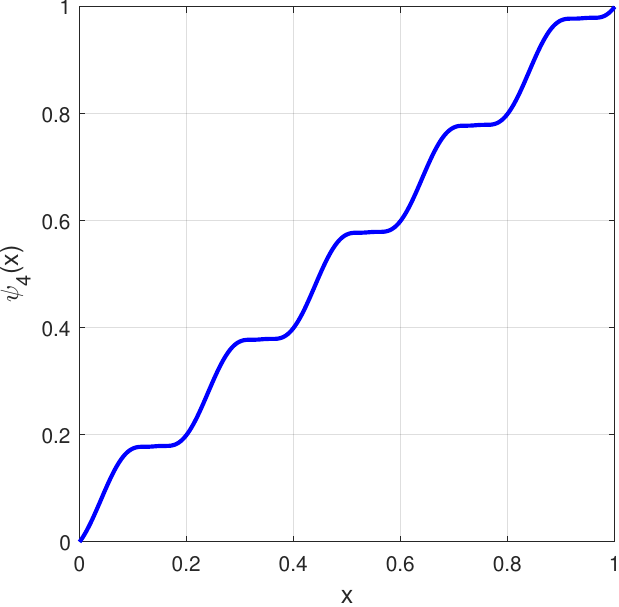}
     } \quad 
     \subfigure[$\psi_5(x)$ at $L_5$]{
    \includegraphics[width=0.25\textwidth]{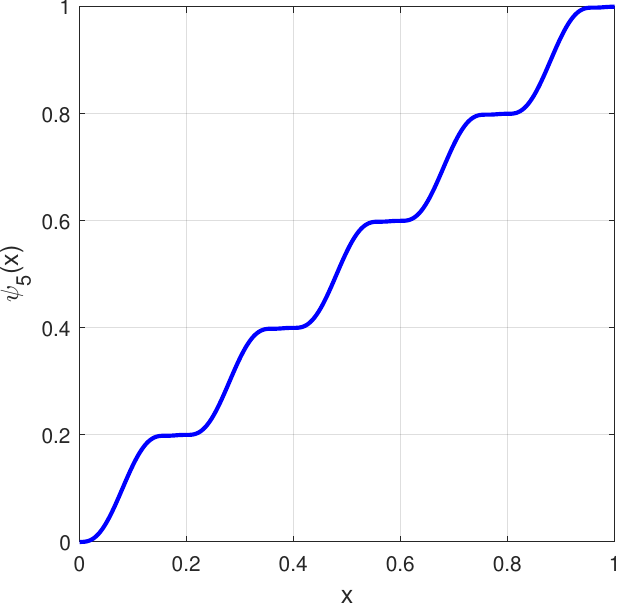}
     } 
             \caption{(a)\, Partition of $[0,1]$ by closed sub-intervals and gaps in \eqref{LqdefnL} with $N=5,\delta=1/25,$ and $\nu=\delta^2$. (b)-(f)\, Plots of $\psi_q(x)$ at $L_q$ with $q=1,\ldots,5.$
             }\label{figsubLq2D}
\end{figure}

\smallskip
Built upon these explicit inner functions, we define the corresponding Kolmogorov maps (see Brattka \cite{Brattka2007} for the definition of this notion).
\begin{definition}[{\bf Kolmogorov maps}]\label{DefnIK} 
\emph{Let $\bs \lambda = (\lambda_1, \ldots, \lambda_d)$ be a given sequence of  positive numbers, and let $\psi_q(x)$ be the  piecewise $C^2$-functions defined in \eqref{PsiqxAS}. Then we define the Kolmogorov maps $\Psi_q:\, [0,1]^d\mapsto  [0,|\bs \lambda|_1]$  as the superpositions:
\begin{equation}\label{d-PsiA0}
   \Psi_q(\bs x):=\Psi_q(\bs x;\bs \lambda):=\sum_{p=1}^d \lambda_p \psi_{q}(x_p), 
\end{equation}
for $q=1,\ldots,  N,$ where $|\bs \lambda|_1=\lambda_1+\cdots+\lambda_d.$}
\end{definition}

\begin{rem}\label{singlephi}
\emph{In keeping with the essence of the exact representations in \cite{Sprecher1965structure, Koppen2002, braun2009application},  the so-defined inner functions  can also be generated from one function {\rm(}i.e., $\phi(t)$ in \eqref{shapfunc0}{\rm)} through translations and vertical shifts:  
\begin{equation}\label{d-PsiAB}
    \Psi_{l+1}(x_1,\ldots,x_d)=\sum_{p=1}^d \lambda_p \big\{\phi(x_p+\alpha l)+\beta_l \big\},\quad l=0,1,\ldots, N-1,
\end{equation}
where $\alpha=-\delta$ and $\beta_l=-\phi(-\delta l).$}\qed 
\end{rem}

\section{Construction of outer functions and approximate version}\label{Sub:center} 

This section is devoted to the explicit construction of the outer functions, thereby leading to the approximate Kolmogorov superpositions.  This requires a careful analysis of the separation properties of the Kolmogorov maps, particularly regarding the selection of the parameters $\bs \lambda$ and $\nu$, to mitigate dislocations caused by these singular maps. 


\subsection{Geometric setup} 
We first introduce the geometric setup based on the partitions in \eqref{1D-P0qq}-\eqref{LqdefnL}. 
Using the sub-intervals $I_q^j$  in \eqref{1D-G0qq}, we form, for each $q$, a set of disjoint hypercubes  
\begin{equation}\label{Cqj}   
 \bs C_q^{\bs j}:= I_q^{j_1}\times \ldots \times   I_q^{j_d},\quad   \bs j=(j_1,\ldots, j_d).
\end{equation} 
Each hypercube has a volume $|\bs C_q^{\bs j}|=(h-\delta)^d,$ and the coordinate of its center is
\begin{equation}\label{cqjcenter}  \bs c_q^{\bs j}:=(c_q^{j_1},\ldots, c_q^{j_d})\;\;\; {\rm with} \;\;\; c_q^{j_p}=j_ph+(h+\delta)/2+(q-1)\delta.
    \end{equation} 
Let
$\bs v_q^{\bs j}, \tilde{\bs v}_q^{\bs j}$ be the coordinates of 
lower-corner (i.e., component-wise minimum vertex) and upper-corner  (i.e., component-wise maximum vertex) 
of  $\bs C_q^{\bs j},$ respectively.  Then we have 
\begin{equation}\label{lower-corner1}
   {\bs v}_{q}^{\bs j}=( v_q^{j_1},\ldots,  v_q^{j_d})=\bs j h+q\,\delta\,\bs 1,\quad \tilde{\bs v}_{q}^{\bs j}=(\tilde v_q^{j_1},\ldots, \tilde v_q^{j_d})=\bs v_q^{\bs j}+(h-\delta)\,\bs 1,
\end{equation}
where $\bs 1=(1,\ldots, 1)\in \mathbb R^d.$ We can also express the coordinate components in  
$t_j, t_j^\delta$  in \eqref{tjhd} as  
$v_q^j=t_j^\delta+(q-1)\delta$ and $\tilde v_q^j=t_{j+1}+(q-1)\delta.$ It is clear that 
\begin{equation}\label{cqjc-00}
\bs c_q^{\bs j}=\frac 1 2 (\bs v_q^{\bs j}+
\tilde {\bs v}_q^{\bs j}),\quad 1\le q\le N.
\end{equation}

We are  interested in the hypercubes  
$\bs C^{\bs j}_q\subseteq [0,1]^d$ and $\bs C_q^{\bs j}\cap [0,1]^d\not=\emptyset,$ which only involve 
$I_q^{j_p}$ with 
$-1\le j_p\le N-1$  in $L_q$ defined 
\eqref{LqdefnL}.  When all $j_p\in {\mathcal J}_{\rm in}:=\{0,1,\ldots, N-2\},$ the hypercube $\bs C_q^{\bs j}$ is  \emph{interior}. When there exists $j_p\in {\mathcal J}_{\rm b}:=\{-1, N-1\},$ the hypercubes $\bs C^{\bs j}_q$ may be intersected by the boundary hyperplanes $x_p=0$ or $x_p=1,$ thereby producing partial (incomplete) hypercubes, which we call \emph{boundary hypercubes}. 
\wll{In such cases, we relocate and redefine the corresponding coordinate components   $v_q^{j_p}=0$ (for all $j_p=-1$) and $\tilde v_q^{j_p}=1$ (for all $j_p=N-1$) so as to include the endpoints $x_p=0,1$.  
In particular,  for $q=1,N,$ the hypercubes $\bs C_q^{\bs j}$  can be treated as interior hypercubes with the corner coordinates given by  \eqref{lower-corner1} with all $j_p\in \{0,\ldots N-1\}$ for $q=1$   (i.e., $\tilde v_q^{-1}= v_q^{-1}=0$)
and  all $j_p\in \{-1,\ldots N-2\}$ for $q=N$ (i.e., $v_q^{N-1}=\tilde v_q^{N-1}=1$).  
With a slight abuse of notation, we still denote these relocated and modified ``corners'' by the same notation. Correspondingly, we redefine the centers of the boundary hypercubes by $c_q^{j_p}=v_q^{j_p}=0$ for $j_p=-1$ at $x_p=0$ and $c_q^{j_p}=\tilde v_q^{j_p}=1$ for $j_p=N-1$ at $x_p=1.$}   In Figure \ref{figsub2D-A0}, we illustrate the hypercubes, centers, and gaps for several levels and also the covering of $[0,1]^2$ by $\delta$-shifting, where $N=5$ and $\delta=1/25$ as in  Figure \ref{figsubLq2D} (a).

\begin{figure}[!h]
\centering
 \subfigure[$L_1$]{
\includegraphics[width=0.25\textwidth]
{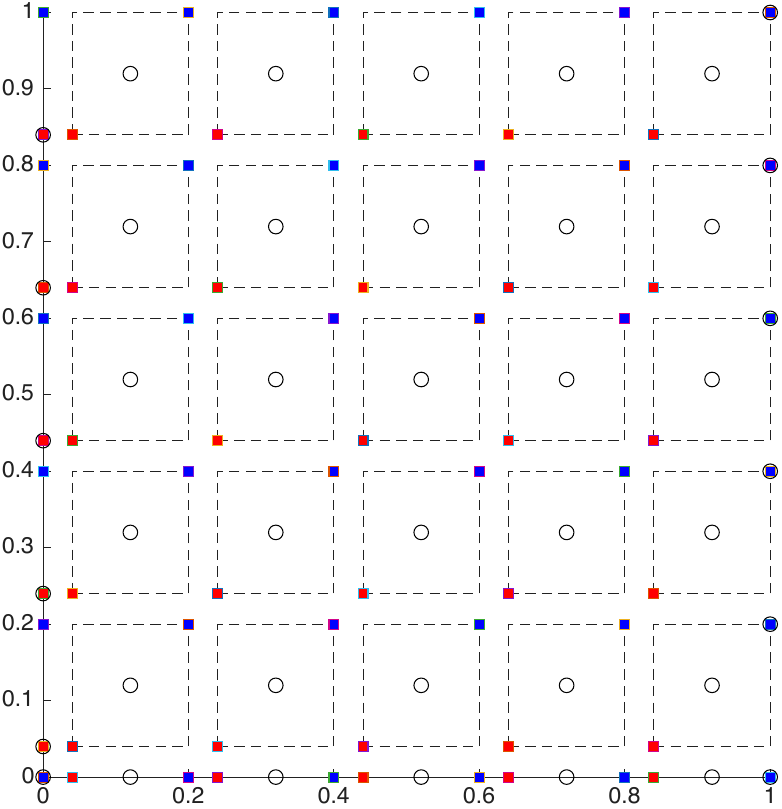}
     } \quad 
      \subfigure[$L_2$]{
\includegraphics[width=0.25\textwidth]
{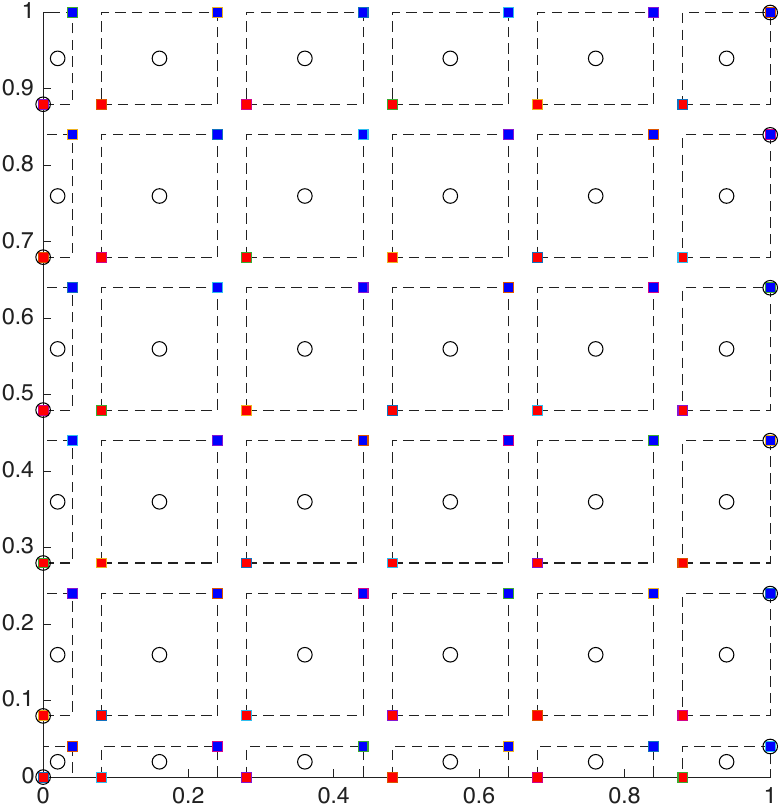}
     }\quad 
        \subfigure[$L_3$]{
\includegraphics[width=0.25\textwidth]
{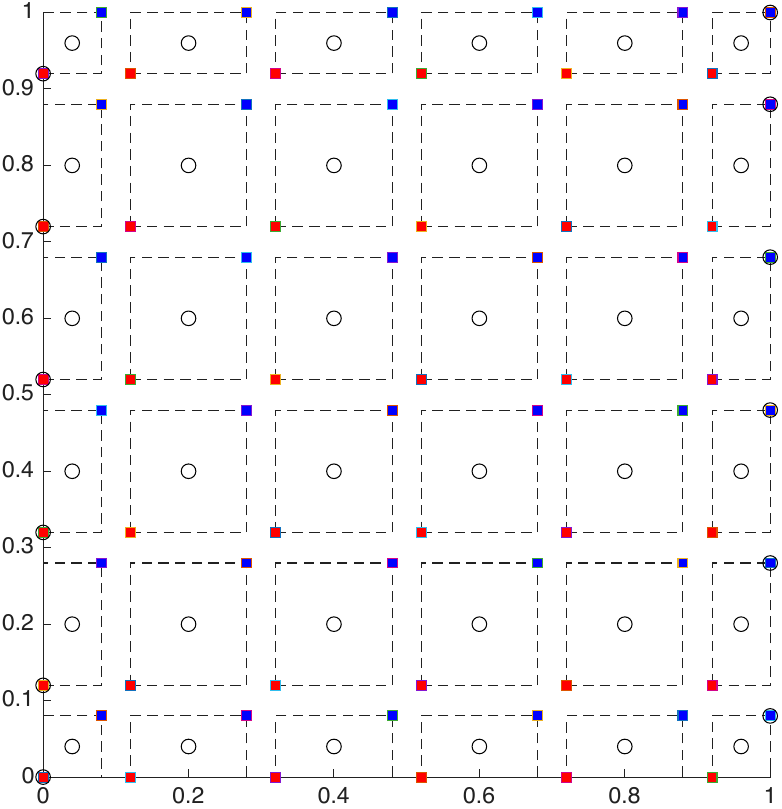}
     } 
    
      \hspace*{10pt}  \subfigure[$L_4 $]{
    \includegraphics[width=0.25\textwidth]
 {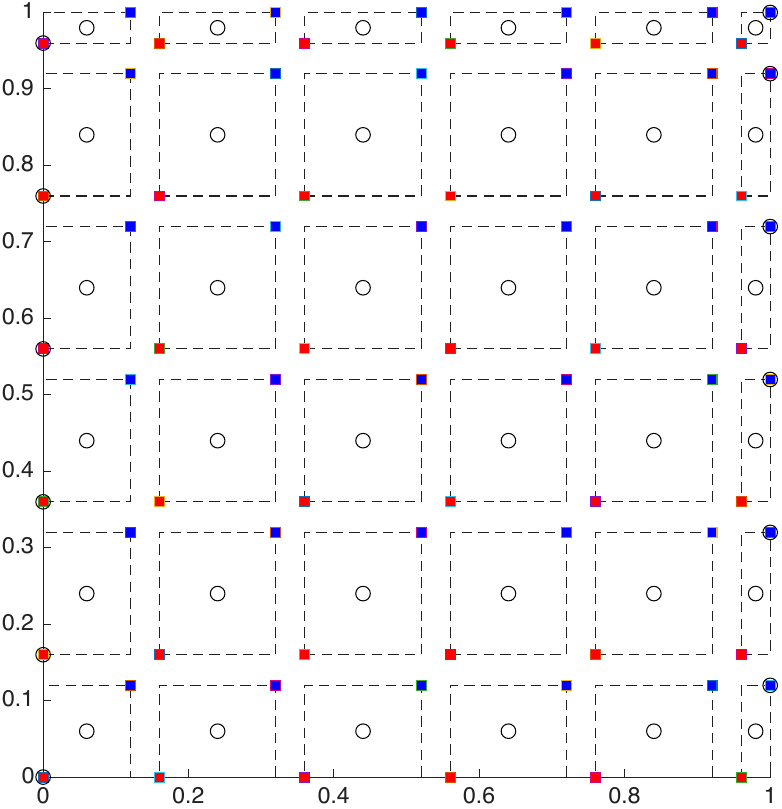}
     }\quad\; 
     \subfigure[$L_5 $]{
    \includegraphics[width=0.25\textwidth]
 {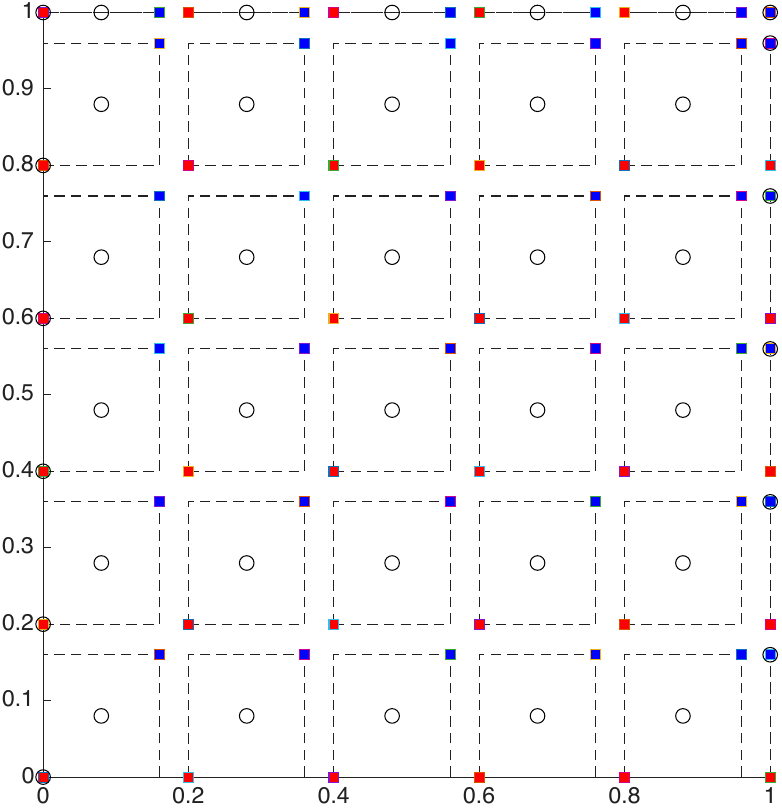}
     }\quad 
     \subfigure[All levels $L_q$]{
    \includegraphics[width=0.27\textwidth]{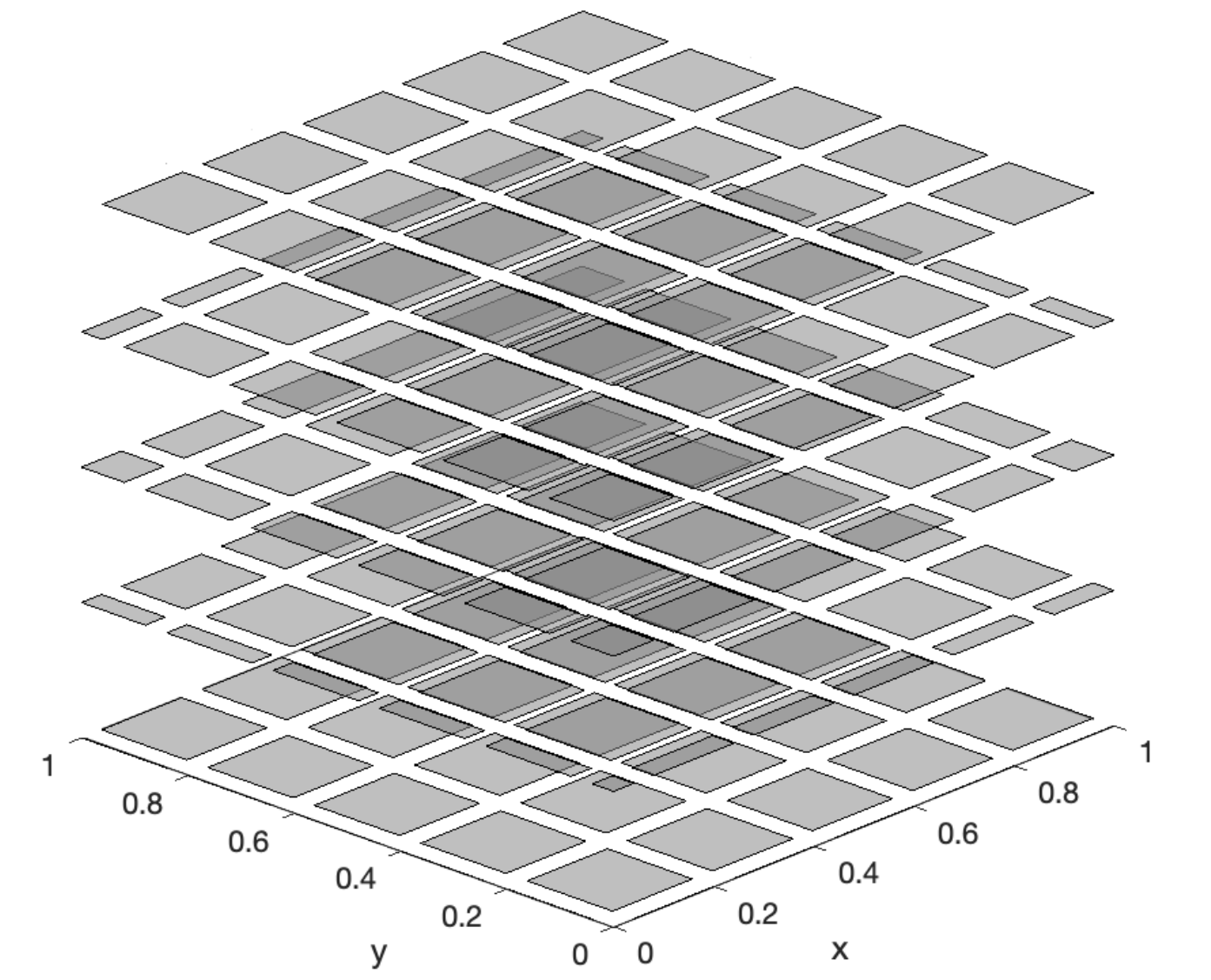}
     }
             \caption{(a)-(e)\, Illustration of  
             disjoint hypercubes,``corners'' (square) ``centers'' (circle) and gaps at levels
             ${L}_q$ with $q=1,\cdots,5$
             in 2D. (f) Covering of $[0,1]^2$ at different levels by $\delta$-shifting. Here, $N=5$ and $\delta=1/25.$}\label{figsub2D-A0}
\end{figure}


       

Correspondingly, we define the mapped hypercube centers, lower corners, and upper corners by
\begin{equation}\label{aqJaj}
z_q^{\bs j}:=\Psi_q(\bs Z_q^{\bs j})=\Psi_q(\bs Z_q^{\bs j};{\bs \lambda})=\sum_{p=1}^d \lambda_p \psi_q(Z_q^{j_p}),\quad \bs j\in \mathcal J^d,\;\; 1\le q\le N,
\end{equation}
where we denote $z_q^{\bs j}=a_q^{\bs j}, e_q^{\bs j}, \tilde e_q^{\bs j}$
for $\bs Z_q^{\bs j}=(Z_q^{j_1},\ldots, Z_q^{j_d})= \bs c_q^{\bs j}, \bs v_q^{\bs j},\tilde{\bs v}_q^{\bs j},$ respectively.
We sample $f\in C([0,1]^d)$ at the hypercube centers and evenly distribute them to $N$ levels:
\begin{equation}\label{aqJaj2-A0}
 b_q^{\bs j}:={f(\bs c_q^{\bs j})}/N,\quad 1\le q\le N.
\end{equation}
The mapped centers 
will serve as the abscissas in $[0,|\bs \lambda|_1]$ for interpolating  
$b_q^{\bs j}$ sampled from the centers $\bs c_q^{\bs j}$ to explicitly construct outer functions. For $d\ge 2$ and fixed $N$ (with $h=1/N$ and $\delta=h^2$), the distributions of  
$\{z_q^{\bs j}\}$ depends on the parameters $\bs \lambda$ and $\nu$. \wll{The critical issue is {\em whether there exist parameters $\bs \lambda$ and $\nu$ such that  $\{a_q^{\bs j}\}$ are well separated and, more importantly, such that whenever two hypercubes $\bs C_q^{\bs j},\, \bs C_q^{\bs j'}$ are neighbors, the corresponding images $z_q^{\bs j},\, z_q^{\bs j'}$ {\rm(}with $z=a,e,\tilde e${\rm)}  remain close to each other.} Otherwise, we say the data are dislocated.}  Much of this section focuses on identifying such parameters so as to avoid or mitigate dislocations.

\subsection{Separation of  hypercubes under the Kolmogorov maps}
For notational convenience, we sometimes denote  
\begin{equation}\label{zZjk} 
z_{q,k}^{\bs j,j_p}:=z_q^{(j_1,\ldots,j_{p-1},k, j_{p+1},\ldots, j_d)},\quad  \bs Z_{q,k}^{\bs j,j_p}:=\bs Z_q^{(j_1,\ldots,j_{p-1},k, j_{p+1},\ldots, j_d)},
\end{equation}
with fixed $j_l,\, l\not=p$ and $j_p=k,$ for 
$z=a,e,\tilde e$ and $\bs Z=\bs c, \bs v, \tilde {\bs v}.$ 
Clearly, by definition (see \eqref{cqjcenter}-\eqref{cqjc-00}), we have $v_q^{j_p}<c_q^{j_p}<\tilde v_q^{j_p}$ for all $1\le p\le d,$ so the strict monotonicity of the inner functions implies that for $\lambda_p>0,$ 
 \begin{equation}\label{lambdapsi} 
 \lambda_p \psi_q(v_q^{j_p}) <\lambda_p \psi_q(c_q^{j_p}) < 
 \lambda_p \psi_q(\tilde v_q^{j_p})\;\;\; \Rightarrow\;\;\; 
 e_q^{\bs j}< a_q^{\bs j}< \tilde e_q^{\bs j}\,.
 \end{equation}
 Moreover, all points $\bs x\in \bs C_q^{\bs j}$ are mapped to the sub-interval $[e_q^{\bs j},\, \tilde e_q^{\bs j}],$ that is,  
\begin{equation}\label{mon-hyperb}
e_q^{\bs j}=\Psi_q(\bs v_q^{\bs j};\bs \lambda)  \le \Psi_q(\bs x;\bs \lambda) \le  \Psi_q(\tilde{\bs v}_q^{\bs j};\bs \lambda)=\tilde e_q^{\bs j},\quad \forall\, \bs x\in \bs C_q^{\bs j}.
\end{equation}
Thus, we can control each hypercube $\bs C_q^{\bs j}$ via the pair $\{e_q^{\bs j},\, \tilde e_q^{\bs j}\}.$  

We also readily find from the strict monotonicity of the inner functions that the mapped centers increase row-wise. For example, for the interior hypercubes, we have
\begin{equation}\label{distanceajj-00}
 a_{q,k+1}^{\bs j,j_p}-a_{q,k}^{\bs j,j_p}=\lambda_p h, \quad k=0,\ldots, N-3.
\end{equation} 
Indeed, by  Definition \ref{DefnIK} and \eqref{cqjcenter}-\eqref{aqJaj}, 
\[
\begin{split}
a_{q,k}^{\bs j,j_p}=\lambda_p\, \psi_q(c_q^{k})+\sum_{l\not=p} \lambda_l\, \psi_q(c_q^{j_l}),\quad 
a_{q,k+1}^{\bs j,j_p}=\lambda_p\, \psi_q(c_q^{k+1})+\sum_{l\not=p} \lambda_l\, \psi_q(c_q^{j_l}),
\end{split}
\]
which yields
\begin{equation*}
\begin{split}
a_{q,k+1}^{\bs j,j_p}-a_{q,k}^{\bs j,j_p} & =\lambda_{p}\,\big(\psi_q(c_q^{k+1})-\, \psi_q(c_q^{k})\big)= \lambda_{p}\,\big(\phi((k+1) h+(h+\delta)/2)\\ &\quad  -\phi(kh+(h+\delta)/2))\big)=\lambda_{p}\,h>0.
\end{split}
\end{equation*}
In view of this, we therefore adopt the row-major ordering of $\bs Z_{q,k}^{\bs j, j_p}$ and $z_{q,k}^{\bs j, j_p}$ in  \eqref{zZjk},   
where one index runs through its range while all other indices are fixed.   
For example, for $d=2,$ we introduce 
\begin{equation}\label{ufform}
\begin{split}
&{\bs A_q}=\big(\underbrace{a_q^{(-1,-1)},\cdots, a_q^{(N-1,-1)}}_{j_2=-1}, \, \underbrace{a_q^{(-1,0)},\cdots, a_q^{(N-1,0)}}_{j_2=0},\cdots,
\underbrace{a_q^{(-1,N-1)},\cdots, a_q^{(N-1,N-1)}}_{j_2=N-1}\big),
\end{split}
\end{equation} 
and 
\begin{align}
{\bs E_q} &=\big(\underbrace{e_q^{(-1,-1)}, \tilde e_q^{(-1,-1)}, e_q^{(0,-1)}, \tilde e_q^{(0,-1)},\cdots, e_q^{(N-2,-1)}, \tilde e_q^{(N-2,-1)}, e_q^{(N-1,-1)}, \tilde e_q^{(N-1,-1)}}_{j_2=-1}, \, \notag\\
&\qquad \underbrace{e_q^{(-1,0)}, \tilde e_q^{(-1,0)}, e_q^{(0,0)}, \tilde e_q^{(0,0)},\cdots, e_q^{(N-2,0)}, \tilde e_q^{(N-2,0)}, e_q^{(N-1,0)}, \tilde e_q^{(N-1,0)}}_{j_2=0}, \notag\\
&\qquad \hspace*{5cm} \vdots \notag \\
&\qquad \underbrace{e_q^{(-1,N-1)}, \tilde e_q^{(-1,N-1)}, e_q^{(0,N-1)}, \tilde e_q^{(0,N-1)}, \cdots, e_q^{(N-1,N-1)}, \tilde e_q^{(N-1,N-1)}}_{j_2=N-1}\big). \label{Efform-order}
\end{align}
Likewise, we define $\bs B_q$ with $b_q^{\bs j}$ in place of $a_q^{\bs j},$ and extend these sequences to higher dimensions.       

We next determine the conditions on $h,\nu$ and $\bs \lambda$ so that the sequences  $\bs A_q, \bs E_q$ are strictly increasing.  
\begin{lemma}\label{lmm:Hypercube-Sep} Let $d\ge 2,$ $h>\nu>0$ and  $\bs \lambda=(\lambda_1, \ldots, \lambda_d)$ be a  sequence of positive real numbers. 
Let $\bs E_q$ be the sequence formed by $\{e_q^{\bs j}, \tilde e_q^{\bs j}\}$ in the order induced by the row-major ordering of the corresponding hypercube corners $\{\bs v_q^{\bs j}, 
\tilde {\bs v}_q^{\bs j}\}$ in \eqref{mon-hyperb}. Then the sequence $\bs E_q$ is strictly increasing, if 
\begin{equation}\label{lmbdcondforDdim}
\begin{split}
\lambda_1 \nu > (h-\nu)\sum_{k=2}^d \lambda_k;
\quad 
\lambda_p \nu >
\sum\limits_{k=1}^{p-1}\lambda_k +(h-\nu)\sum\limits_{\ell=p+1}^d \lambda_\ell,\;\;\; 2\le p\le d,
\end{split}
\end{equation}   
so is the sequence $\bs A_q$ formed by arranging the mapped centers
$\{a_q^{\bs j}\}$ in row-major order.  
\end{lemma}
\begin{proof}
In view of \eqref{mon-hyperb}, if $\lambda_p>0$ and $h>\nu>0,$ then we have 
$e_q^{\bs j}<\tilde e_q^{\bs j}.$ Thus, it suffices to identify conditions ensuring that, for any neighboring pair  $\{\tilde e_q^{\bs j}, e_q^{\bs j'}\}\in \bs E_q,$ one has 
$\tilde e_q^{\bs j}<e_q^{\bs j'}.$ Moreover, by \eqref{lambdapsi}, the strict monotonicity of $\bs A_q$ is an immediate consequence.  

For clarity, we first consider $d=2$ and examine the mapped corners of the hypercubes in the same row (i.e., with fixed $j_2$) in  \eqref{Efform-order}. We find from
 \eqref{intercond}, \eqref{phiqibnd}, Definition \ref{Inner-fun},
 \eqref{lower-corner1} and \eqref{mon-hyperb} that\\[4pt] 
 (i) for $0\leq j_2\leq N-2$  (the corner coordinates given by \eqref{lower-corner1}),  
\begin{equation}\label{sepneighbsupercubes0}
\begin{split}
 e_q^{(j_1+1,j_2)}-\tilde e_q^{(j_1,j_2)}&=
\lambda_1(\psi_q(v_q^{j_1+1})-\psi_q(\tilde v_q^{j_1}))
+\lambda_2(\psi_q( v_q^{j_2})-\psi_q(\tilde v_q^{j_2}))
\\&
=\lambda_1(\phi(t_{j_1+1}^\delta)-\phi(t_{j_1+1}))
+\lambda_2(\phi(t_{j_2}^\delta)-\phi(t_{j_2+1}))
\\&=\lambda_1\nu +\lambda_2 (\nu-h);
\end{split}
\end{equation}
(ii) for $j_2=-1$ (i.e., $v_q^{j_2}=0$),   
\begin{equation}\label{sepneighbsupercubes-1}
\begin{split}
 e_q^{(j_1+1,-1)}&-\tilde e_q^{(j_1,-1)}
=\lambda_1(\psi_q( v_q^{j_1+1})-\psi_q(\tilde v_q^{j_1}))
+\lambda_2(\psi_q( v_q^{-1})-\psi_q(\tilde v_q^{-1}))
\\&
=\lambda_1(\phi(t_{j_1+1}^\delta)-\phi(t_{j_1+1}))
+\lambda_2(-(\phi(0)-\phi((1-q)\delta)))
\\& =\lambda_1\nu +\lambda_2 \phi((1-q)\delta)\geq \lambda_1\nu +\lambda_2 (\nu-h);
\end{split}
\end{equation}
 (iii)
for 
$j_2=N-1$ (i.e., $\tilde v_q^{j_2}=1$), 
\begin{equation}\label{sepneighbsupercubesJ-1}
\begin{split}
 e_q^{(j_1+1,N-1)} &-\tilde e_q^{(j_1,N-1)}
=\lambda_1(\psi_q( v_q^{j_1+1})-\psi_q(\tilde v_q^{j_1}))
+\lambda_2(\psi_q( v_q^{N-1})-\psi_q(\tilde v_q^{N-1}))
\\&
=\lambda_1(\phi(t_{j_1+1}^\delta)-\phi(t_{j_1+1}))
+\lambda_2(\phi(t_{N-1}^{\delta})-\phi((1-q)\delta)-\phi(1))
\\& 
=\lambda_1\nu + \lambda_2 (\nu-h-\phi((1-q)\delta))\\
& \ge \lambda_1\nu +\lambda_2 (\nu-h),
\end{split}
\end{equation}
where $-1\le j_1\le N-2$ and $1\le q\le N$ in   
\eqref{sepneighbsupercubes0}-\eqref{sepneighbsupercubesJ-1}. Thus,  we require
$
{\lambda_1}\nu > (h-\nu)\lambda_2$ to separate the hypercubes in each row.  We next ensure the tails  and heads for two consecutive rows
with $j_2, j_2+1$ in $\bs E_q$ are in strictly increasing order, which requires
\begin{equation}\label{sepneighbrow}
\begin{split}
 e_q^{(-1,j_2+1)} -\tilde e_q^{(N-1,j_2)}&=\lambda_1(\psi_q( v_q^{-1})-\psi_q(\tilde v_q^{N-1}))
+\lambda_2(\psi_q( v_q^{j_2+1})-\psi_q(\tilde v_q^{j_2}))\\&
=\lambda_1(\phi(0)-\phi(1))
+\lambda_2(\phi(t_{j_2+1}^{\delta})-\phi(t_{j_2+1}))\\
&=-\lambda_1+\lambda_2\nu> 0, 
\end{split}
\end{equation}
for $-1\leq j_2\leq N-2,$ 
 which implies  ${\lambda_2}\nu>{\lambda_1}.$  Thus, a combination of the above two conditions leads to \eqref{lmbdcondforDdim} with $d=2.$

Similarly, in the three-dimensional case, we arrange the mapped corners of the hypercubes in row-major ordering as 
\begin{equation*}\label{ufform3D}
\begin{split}
{\bs E_q} &=\big(\underbrace{ e_q^{(-1,-1,-1)},\tilde e_q^{(-1,-1,-1)},\cdots,  e_q^{(N-1,-1,-1)},\tilde e_q^{(N-1,-1,-1)}}_{j_2=j_3=-1}, \\[6pt]& \qquad  \underbrace{e_q^{(-1,0,-1)},\tilde e_q^{(-1,0,-1)},\;\; \cdots,\;\;  e_q^{(N-1,0,-1)},\tilde e_q^{(N-1,0,-1)}}_{j_2=0,\, j_3=-1},\; \cdots,\; \\[6pt]&
\qquad 
 \underbrace{ e_q^{(-1,N-1,N-1)},\tilde e_q^{(-1,N-1,N-1)},\cdots,  e_q^{(N-1,N-1,N-1)},\tilde e_q^{(N-1,N-1,N-1)}}_{j_2=j_3=N-1}\big)\in {\mathbb R}^{(2(J+1))^3},
\end{split}
\end{equation*}
for $2\le q\le N-1.$ Likewise, for $q=1,N,$ all $J^3$ hypercubes are interiors, so we can define ${\bs E_q}$ by removing elements and rows as in two dimensions. 
Like \eqref{sepneighbsupercubes0}-\eqref{sepneighbsupercubesJ-1}, we first consider  
the mapped corners of the hypercubes in the same row (i.e.,  for fixed $j_2, j_3$) and  obtain 
\begin{equation}\label{3Dsepneighbsupercubes0}
\begin{split}
& e_q^{(j_1+1,j_2,j_3)}-\tilde e_q^{(j_1,j_2,j_3)}\\& =
\lambda_1(\psi_q( v_q^{j_1+1})-\psi_q(\tilde v_q^{j_1}))
+\lambda_2(\psi_q( v_q^{j_2})-\psi_q(\tilde v_q^{j_2}))+\lambda_3(\psi_q( v_q^{j_3})-\psi_q(\tilde v_q^{j_3}))
\\[4pt]&=
\begin{dcases}
\lambda_1\nu +(\lambda_2+\lambda_3) (\nu-h), & 0\leq j_2,j_3\leq N-2,\\
\lambda_1\nu +(\lambda_2+\lambda_3)\phi((1-q)\delta),& j_2=j_3=-1,\\
\lambda_1\nu +\lambda_2\phi((1-q)\delta)+\lambda_3(\nu-h-\phi((1-q)\delta)),\quad & j_2=-1,j_3=N-1,\\
\lambda_1\nu +\lambda_2(\nu-h-\phi((1-q)\delta))+\lambda_3\phi((1-q)\delta),& j_2=N-1,j_3=-1,\\
\lambda_1\nu +(\lambda_2+\lambda_3) (\nu-h-\phi((1-q)\delta)),& j_2=j_3=N-1,
\end{dcases}\\[4pt]
&\geq \lambda_1\nu +(\lambda_2+\lambda_3) (\nu-h),
\end{split}
\end{equation}
which is exactly \eqref{sepneighbsupercubes0}–\eqref{sepneighbsupercubesJ-1} with $\lambda_2+\lambda_3$ in place of $\lambda_2$.  Thus, each row in $\bs E_p$ is increasing if the first condition of \eqref{lmbdcondforDdim} holds.

We next ensure the tails and heads for two consecutive rows
with $j_2, j_2+1$ and $j_3, j_3+1$ in $\bs E_q$ are in strictly increasing order. Indeed, we find \\[4pt]
(i) for any $j_2\in \{-1,\ldots, N-2\}$,  
\begin{equation*}\label{bqheadtail3d}
\begin{split}
& e_q^{(-1,j_2+1,j_3)}-\tilde e_q^{(N-1,j_2,j_3)}\\
&=\lambda_1(\psi_q( v_q^{-1})-\psi_q(\tilde v_q^{N-1}))
+\lambda_2(\psi_q( v_q^{j_2+1})-\psi_q(\tilde v_q^{j_2}))+\lambda_3(\psi_q( v_q^{j_3})-\psi_q(\tilde v_q^{j_3}))
\\[6pt]&=
\begin{dcases}
-\lambda_1+\lambda_2\nu+\lambda_3(\nu-h), & 0\leq j_3\leq N-2,\\[4pt]
-\lambda_1+\lambda_2\nu+\lambda_3\phi((1-q)\delta),& j_3=-1,\\[4pt]
-\lambda_1+\lambda_2\nu+\lambda_3 (\nu-h-\phi((1-q)\delta)),\quad & j_3=N-1,
\end{dcases}\\[4pt]
&\geq-\lambda_1+\lambda_2\nu+\lambda_3(\nu-h),
\end{split}\end{equation*}
which is positive, if ${\lambda_2}\nu>{\lambda_1+\lambda_3(h-\nu)}$.\\[4pt]
(ii) for any  
$j_3\in \{-1,\ldots, N-2\}$, 
\begin{equation*}\begin{split}
& e_q^{(-1,-1,j_3+1)}-\tilde e_q^{(N-1,N-1,j_3)}\\
&=\lambda_1(\psi_q( v_q^{-1})-\psi_q(\tilde v_q^{N-1}))
+\lambda_2(\psi_q( v_q^{-1})-\psi_q(\tilde v_q^{N-1}))+\lambda_3(\psi_q( v_q^{j_3+1})-\psi_q(\tilde v_q^{j_3}))\\&
=\lambda_1(\phi(0)-\phi(1))
+\lambda_2(\phi(0)-\phi(1))+\lambda_3(\phi(t_{j_3+1}^{\delta})-\phi(t_{j_3+1}))\\
&=-(\lambda_1+\lambda_2)+\lambda_3\nu> 0,
\end{split}\end{equation*}
which implies the condition $\lambda_3\nu> \lambda_1+\lambda_2.$
Combining all conditions 
yields   \eqref{lmbdcondforDdim} with $d=3.$

Finally, for $d\ge 4,$ we need to impose conditions to  guarantee 
the mapped corners of the hypercubes in the same row (i.e.,  for fixed $j_2, j_3,\ldots,j_d$) in strictly increasing order as an extension of \eqref{3Dsepneighbsupercubes0}:
\begin{equation}\label{eqanydd}
\begin{split}
e_q^{(j_1+1,j_2,j_3\ldots,j_d)}-\tilde e_q^{(j_1,j_2,j_3\ldots,j_d)}
&\geq \lambda_1\nu +(\lambda_2+\cdots+\lambda_d)(\nu-h).
\end{split}
\end{equation}
Similar to (i)-(ii) in three dimensions, we also need to make sure  the mapped tails and heads for two consecutive rows are in strictly increasing order:
\begin{equation*}\label{2maincondforanyd}
\begin{dcases}
\tilde e_q^{(N-1,j_2,j_3\ldots,j_d)}<
 e_q^{(-1,j_2+1,j_3,\ldots,j_d)}, & -1\le j_2\le  N-2,   \\
\tilde e_q^{(N-1,N-1,j_3\ldots,j_d)}<
 e_q^{(-1,-1,j_3+1,\ldots,j_d)}, & -1\le j_3\le  N-2,  \\
\qquad \qquad \vdots\\
\tilde e_q^{(N-1,N-1,\ldots,N-1,j_d)}<
 e_q^{(-1,-1,\ldots,-1,j_d+1)}, & -1\le j_d\le  N-2,\end{dcases}
\end{equation*}
where the other indices for each case are fixed, if not specified. Correspondingly,  direct calculation as before leads to 
\begin{equation}\label{secondcondforanydd}
\begin{dcases}
 e_q^{(-1,j_2+1,j_3,\ldots,j_d)}-\tilde e_q^{(N-1,j_2,j_3\ldots,j_d)}
>-\lambda_1+\lambda_2\nu+(\lambda_3+\cdots+\lambda_d)(\nu-h),
  \\[6pt]
e_q^{(-1,-1,j_3+1,\ldots,j_d)}-\tilde e_q^{(N-1,N-1,j_3\ldots,j_d)}>-(\lambda_1+\lambda_2)+\lambda_3\nu+(\lambda_4+\cdots+\lambda_d)(\nu-h), \\
\qquad \qquad \vdots\\
e_q^{(-1,-1,\ldots,-1,j_d+1)}-\tilde e_q^{(N-1,N-1,\ldots,N-1,j_d)}>-(\lambda_1+\lambda_2+\ldots+\lambda_{d-1})+\lambda_d\nu.
\end{dcases}
\end{equation}
Thus, we infer from 
\eqref{eqanydd}-\eqref{secondcondforanydd} the conditions summarized in \eqref{lmbdcondforDdim}. This completes the proof.
\end{proof}


Without loss of generality, we assume that $|\bs \lambda|_1 = 1$, as this normalization can always be attained by replacing $\lambda_p$ with $\lambda_p / |\bs \lambda|_1$ whenever $|\bs \lambda|_1 \ne 1$. 
Defining
\begin{equation}\label{SpQp}
  {\mathcal S}_p=\sum_{k=p}\lambda_k,\quad 1\le p\le d \quad{\rm with}\quad {\mathcal S}_1=|\bs \lambda|_1=1,  
\end{equation}
we have
\begin{equation}\label{lmdpS-A}
    \lambda_p={\mathcal S}_p-{\mathcal S}_{p+1},\quad 1\le p\le d-1;\quad \lambda_d={\mathcal S}_d.
\end{equation}
As $|\bs \lambda|_1=1,$ we can reformulate the condition \eqref{lmbdcondforDdim} as 
\begin{equation}\label{lmbdcondforDdimA}
\begin{split} 
\lambda_p >\frac 1{\nu+1} + 
\dfrac{h-\nu-1}{\nu+1} \sum_{k=p+1}^d\lambda_k,\quad  1\le p\le d,   
\end{split}
\end{equation}  
where  for $p=1,d,$ we understand 
\begin{equation}\label{lmbd1D}
\lambda_1 >\frac {h-\nu} h,\quad \lambda_d >\frac 1{\nu+1}.
\end{equation}
Then  \eqref{lmbdcondforDdimA} is equivalent to 
\begin{equation}\label{lambdapS}
{\mathcal S}_{p}-{\mathcal S}_{p+1}=\lambda_p>\frac{1}{\nu+1}-\frac{\nu+1-h}{\nu+1}{\mathcal S}_{p+1},
\end{equation}
for $1 \le p\le d-1,$ which equivalently yields 
\begin{equation}\label{Sp-recurr}
{\mathcal S}_p>\frac{h}{\nu+1}{\mathcal S}_{p+1}+\frac{1}{\nu+1} \quad \Leftrightarrow  \quad {\mathcal S}_{p+1}<\frac{\nu+1}{h}{\mathcal S}_p-\frac{1}{h}.
\end{equation}

Correspondingly, for fixed $h>0$ and any $\nu\in (0,h),$ we introduce the admissible set 
\begin{equation}\label{admiss-lam-S}
\begin{split}
  {\mathbb A}_h(\nu)&:=\big\{\bs \lambda\in \mathbb R_+^d\,:\, \bs \lambda ~\, \text{satisfies} \,~\, \eqref{lmbdcondforDdim} ~\, \text{and} ~ |\bs \lambda|_1=1  \big\}\\
  &= \big\{\bs {\mathcal S}\in \mathbb R_+^d\,:\, \bs {\mathcal S} ~ \text{satisfies} ~ \eqref{Sp-recurr}~ \text{and}~ {\mathcal S}_1=1  \big\},
  \end{split}
\end{equation}
where $\mathbb R_+$ denotes the set of all positive real numbers and $\bs {\mathcal S}=({\mathcal S}_1,\ldots, {\mathcal S}_d).$
In practical construction, we need to explicitly choose  $\bs\lambda.$ To this end, we first determine the range of ${\mathcal S}_p$ for any $\nu\in (0,h).$ More precisely, we can show that $Q_p<{\mathcal S}_p<\widetilde Q_p,$ where 
\begin{equation}\label{Q-formula}
Q_p=\frac 1 \zeta-\frac 1\zeta \Big(\frac{h}{\nu+1}\Big)^{d-p+1}, 
\quad\;  \widetilde Q_p=
\frac 1 \zeta -\frac{1-\zeta}\zeta  \Big(\frac{h}{\nu+1}\Big)^{1-p},
\end{equation} 
and $\zeta:=\nu+1-h.$ Indeed, corresponding to \eqref{Sp-recurr}, we introduce the backward and forward recurrence relations:
\begin{equation}\label{QpQp}
    Q_p=\frac{h}{\nu+1}Q_{p+1}+\frac{1}{\nu+1}; \quad 
    \widetilde Q_{p+1}=\frac{\nu+1}{h}\widetilde Q_{p}-\frac 1 h,
\end{equation}
 with the initial values
$Q_d:=\tfrac{1}{\nu+1}$ and $\widetilde Q_1:=1,$ respectively. We can derive \eqref{Q-formula} from induction straightforwardly.  
We further verify $Q_p<{\mathcal S}_p<\widetilde Q_p$ by induction. Clearly,  
for $p=d$, ${\mathcal S}_d=\lambda_d>\tfrac{1}{\nu+1}=Q_d.$ 
For $1\leq p\leq d-1$, we assume ${\mathcal S}_{p+1}>Q_{p+1}$ holds, then  by \eqref{Sp-recurr}, 
\begin{equation*}
{\mathcal S}_p>\frac{1}{\nu+1}+\frac{h}{\nu+1}{\mathcal S}_{p+1}>\frac{1}{\nu+1}+\frac{h}{\nu+1}Q_{p+1}=Q_p.
\end{equation*}
Similarly, we can show ${\mathcal S}_p<\widetilde Q_p,$  leading to the bounds of ${\mathcal S}_p$ in  \eqref{Q-formula}.  
In particular, for $p=2,$  we obtain immediately that 
\begin{equation*}
Q_2=\frac{1-(\frac{h}{\nu+1})^{d-1}}{\nu+1-h}<{\mathcal S}_2<\frac{\nu}{h}=\widetilde Q_2.
\end{equation*}
We take $\nu$ as the positive root of the quadratic  equation: 
\begin{equation}\label{mueqn} 
    \frac{1-(\frac{h}{h+1})^{d-1}}{\nu+1-h}=\frac{\nu}{h}, 
\end{equation}
that is,  
\begin{equation}\label{qAeqnA} 
 \nu=\nu^*:=-\frac{1-h}{2}+ \frac{1-h}{2}\sqrt{1+ \frac 
    {4h}{(1-h)^2}\Big(1 - \Big(\frac{h}{h+1}\Big)^{d-1}\Big)}\,,
\end{equation}
which allows us to explicitly determine a set of  
$\{\lambda_p^*\}_{p=1}^d$ for practical use.   For clarity, for $\nu=\nu^*,$
we denote ${\mathcal S}_p^*={\mathcal S}_p,\, Q_p^*=Q_p,$  and likewise for other notation.

\begin{theorem}\label{Hypercube-Sep} For $d\ge 2$ and   $0<h<1,$ we take $\nu=\nu^*$ in \eqref{qAeqnA}, and set 
\begin{equation}\label{Spcurrent}
{\mathcal S}_1^*:=1,\quad {\mathcal S}_{p+1}^*=\frac{1}{2}\Big(Q^*_{p+1}+\frac{\nu^*+1}{h}{\mathcal S}_p^*-\frac{1}{h}\Big),\quad 1\le p\le d-1,
\end{equation}
with $Q_d^*=\tfrac 1 {\nu^*+1}.$ Letting 
\begin{equation}\label{lmdpS}
    \lambda_p^*={\mathcal S}_p^*-{\mathcal S}_{p+1}^*,\quad 1\le p\le d-1;\quad \lambda_d^*={\mathcal S}_d^*,
\end{equation}
we have $\bs \lambda^*\in {\mathbb A}_h(\nu^*),$ and 
\begin{equation}\label{def_d_lambda}
    \lambda_p^*=
        \frac{h^{d-p}}{(\nu^*+1)^{d-p+1}}+\chi_p\frac{\nu^*-h+(\frac{h}{\nu^*+1})^d}{\nu^*+1-h}\Big(\frac{\nu^*+1}{2h}\Big)^{p-1},
\end{equation}
where $\chi_p=\tfrac{2h-\nu^*-1}{2h}$ for $1\le p\le d-1$ and $\chi_p=1$ for $p=d$.
With this choice, the corresponding sequence $\bs E_q$ is strictly increasing. 
\end{theorem}
\begin{proof}
We first show that  $\nu^*\in (0,h).$ We rewrite \eqref{mueqn} as
$$
Q(\nu):=\nu^2 +(1-h)\nu- h+h   (\tfrac{h}{h+1})^{d-1}=0,
$$
It is clear that $Q(0)<0,\, Q(h)>0$ and $Q(h-1)>0,$ so 
$Q(\nu)$ has a unique positive root in $(0,h).$
With the choice \eqref{Spcurrent}, we have 
\begin{equation*}
\begin{split}
{\mathcal S}_{p+1}^*&=\frac{1}{2}\Big(Q^*_{p+1}+\frac{\nu^*+1}{h}{\mathcal S}_p^*-\frac{1}{h}\Big)=\frac{1}{2}\Big(\frac{\nu^*+1}{h}Q^*_p-\frac{1}{h}+\frac{\nu^*+1}{h}{\mathcal S}_p^*-\frac{1}{h}\Big)\\
&<\frac{\nu^*+1}{h}{\mathcal S}_p^*-\frac{1}{h},
\end{split}
\end{equation*}
which implies $\bs S^*\in {\mathbb A}_h(\nu^*)$ in
\eqref{admiss-lam-S}, so is $\bs \lambda^*.$  
Moreover,  using \eqref{Spcurrent}  and the first formula in
\eqref{QpQp}, we find
\begin{equation*}
\begin{aligned}
{\mathcal S}_{p+1}^*-Q_{p+1}^*&=\frac{1}{2}\bigg(\frac{\nu^*+1}{h}Q^*_p-\frac{1}{h}+\frac{\nu^*+1}{h}{\mathcal S}_p^*-\frac{1}{h}\bigg)-\bigg(\frac{\nu^*+1}{h}Q^*_p-\frac{1}{h}\bigg) \\
&=\frac{\nu^*+1}{2h}({\mathcal S}_{p}^*-Q_{p}^*),
\end{aligned}
\end{equation*}
which implies
$${\mathcal S}_{p}^*-Q_{p}^*=\Big(\frac{\nu^*+1}{2h}\Big)^{p-1}({\mathcal S}_1^*-Q_1^*)=\Big(\frac{\nu^*+1}{2h}\Big)^{p-1}(1-Q_1^*).$$
Then by \eqref{lmdpS},
we compute 
\begin{equation}
\begin{aligned}
\lambda_p^* &={\mathcal S}_p^*-{\mathcal S}_{p+1}^*\\
&=(Q_p^*-Q_{p+1}^*)+({\mathcal S}_p^*-Q_p^*)-({\mathcal S}_{p+1}^*-Q_{p+1}^*)\\
&=(Q_p^*-Q_{p+1}^*)+(\tfrac{\nu^*+1}{2h})^{p-1}(1-Q_1^*)-(\tfrac{\nu^*+1}{2h})^{p}(1-Q_1^*)\\
&=(Q_p^*-Q_{p+1}^*)+\tfrac{2h-\nu^*-1}{2h}(\tfrac{\nu^*+1}{2h})^{p-1}(1-Q_1^*).
\end{aligned}
\end{equation}
for $1\le p\le d-1$, and
$$\lambda_d^*={\mathcal S}_d^*=Q_d^*+({\mathcal S}_d^*-Q_d^*)=Q_d^*+(\tfrac{\nu^*+1}{2h})^{d-1}(1-Q_1^*).$$
Inserting the expressions of $Q_p=Q_p^*$  in \eqref{Q-formula} into the above, we obtain \eqref{def_d_lambda} and complete the proof.
\end{proof} 

\begin{corollary}\label{lamda-small} Let $\nu^*$ and  $\bs \lambda^*=(\lambda_1^*,\ldots,\lambda_d^*)$ be given in \eqref{qAeqnA} and \eqref{def_d_lambda}, respectively. Then for  $0<h\ll 1,$  we have the asymptotics
 \begin{equation}\label{nu-asym}
 \nu^*=h-h^d+O(h^{d+1}), 
 \end{equation}
 and 
 \begin{equation}\label{lamda-asym}
 \lambda_p^*=h^{d-p}+O(h^{d-p+1}),\quad  1\le p\le d-1;\quad  
 \lambda_d^*=1-h+O(h^2).
 \end{equation}
\end{corollary}
\begin{proof} Define $\varepsilon=\varepsilon(h)=h-\nu^*.$ Then from \eqref{mueqn}, we obtain the equation of $\varepsilon:$
$$(1+h)^{d-1}\,\varepsilon^2-(1+h)^d\,\varepsilon+h^d=0.$$
It is clear that  \(\varepsilon(0)=0\). Differentiating this identity \(k\) times and taking \(h=0\) for \(1\le k\le d-1\), 
 we have $\varepsilon^{(k)}(0)=0$,
while differentiating \(d\) times gives
$\varepsilon^{(d)}(0)=d!.$
Thus, using Taylor's theorem yields
$\varepsilon=h^d+O(h^{d+1})$
and \eqref{nu-asym}.

We now examine $\bs \lambda^*$. The first term of \eqref{def_d_lambda} is $h^{d-p}+O(h^{d-p+1})$; and, in particular, for $p=d,$ it becomes $\tfrac{1}{\nu^*+1}=1-h+O(h^2).$ 
As $\nu^*-h=-h^d+O(h^{d+1})$, we have
$$\frac{\nu^*-h+(\frac{h}{\nu^*+1})^d}{\nu^*+1-h}
=\frac{-h^d+O(h^{d+1})+h^d(1-dh+O(h^2))}{\nu^*+1-h}=O(h^{d+1}),$$
Therefore, the second term in \eqref{def_d_lambda} is $O(h^{d-p+1})$ for $1\le p\le d-1$ and $O(h^2)$ when $p=d$.
\end{proof}

\begin{rem}\label{Cond-lambda} {\em In fact, if one only requires $\bs A_q$ to be increasing, the condition
\eqref{lmbdcondforDdim} can be substantially relaxed, especially the constraint
\eqref{nu-asym}. Coincidentally, the choice
$\lambda_p^*=O(N^{p-d})$ is analogous to that in Sprecher
\cite{Sprecher1997numerical}, although the construction of the inner functions
therein, based on K\"oppen's recursive relation, is quite different from ours. \qed}
\end{rem}

\wll{We provide in Table \ref{nulamba-A} some samples of $\nu^*$ and $\bs \lambda^*$ for 
$d=5$ and various $N$ with $h=1/N$ computed from \eqref{qAeqnA}  and \eqref{def_d_lambda}.}
\begin{table}[htbp]
\small
\centering
\caption{\small Values of $\nu^*$ and $\bs \lambda^*$ for $d=5$ and various $N$.}\label{nulamba-A}
\setlength{\tabcolsep}{10pt}
\resizebox{\textwidth}{!}{%
\begin{tabular}{c p{4.0cm} p{4.0cm} p{4.0cm}}
\hline
$N$ & \centering $\nu^*$ & \centering $\lambda_1^*$ & \centering\arraybackslash $\lambda_2^*$ \\
\hline
32   & \centering $3.124997444768155\mathrm{e}{-2}$   & \centering $8.176742323309422\mathrm{e}{-7}$   & \centering\arraybackslash ${2.698324964602761\mathrm{e}{-5}}$ \\
64   & \centering ${1.562499913814696\mathrm{e}{-2}}$   & \centering $5.515859456731126\mathrm{e}{-8}$   & \centering\arraybackslash $3.585308646827693\mathrm{e}{-6}$ \\
128  & \centering $7.812499972006871\mathrm{e}{-3}$  & \centering $3.583120537080563\mathrm{e}{-9}$   & \centering\arraybackslash $4.622225492832924\mathrm{e}{-7}$ \\
256  & \centering ${3.906249999108063\mathrm{e}{-3}}$  & \centering $2.283359788964226\mathrm{e}{-10}$  & \centering\arraybackslash $5.868234657638041\mathrm{e}{-8}$ \\
512  & \centering ${1.953124999971854\mathrm{e}{-3}}$  & \centering $1.441063556854805\mathrm{e}{-11}$  & \centering\arraybackslash $7.392656046665147\mathrm{e}{-9}$ \\
\hline
\end{tabular}%
}


\resizebox{\textwidth}{!}{%
\begin{tabular}{c p{4.0cm} p{4.0cm} p{4.0cm}}
\hline
$N$ & \centering $\lambda_3^*$ & \centering $\lambda_4^*$ & \centering\arraybackslash $\lambda_5^*$ \\
\hline
32   & \centering ${8.904472269424150\mathrm{e}{-4}}$  & \centering ${2.938475793733965\mathrm{e}{-2}}$   & \centering\arraybackslash ${9.696969939118396\mathrm{e}{-1}}$ \\
64   & \centering ${2.330450619433748\mathrm{e}{-4}}$  & \centering ${1.514792901662832\mathrm{e}{-2}}$   & \centering\arraybackslash ${9.846153854541869\mathrm{e}{-1}}$ \\
128  & \centering ${5.962670885671015\mathrm{e}{-5}}$ & \centering ${7.691845442354954\mathrm{e}{-3}}$  & \centering\arraybackslash ${9.922480620431185\mathrm{e}{-1}}$ \\
256  & \centering ${1.508136307012304\mathrm{e}{-5}}$ & \centering ${3.875910309019035\mathrm{e}{-3}}$  & \centering\arraybackslash ${9.961089494172283\mathrm{e}{-1}}$ \\
512  & \centering $3.792432551939167\mathrm{e}{-6}$ & \centering ${1.945517899144752\mathrm{e}{-3}}$  & \centering\arraybackslash ${9.980506822612366\mathrm{e}{-1}}$ \\
\hline
\end{tabular}%
}
\end{table}

\subsection{Construction of outer functions and approximate version} 
Now, we are ready to construct the outer functions. Given any $f\in C([0,1]^d),$   the multi-dimensional sampled data are processed as follows 
  \begin{equation}\label{flowmap-A1}
\underbrace{\big\{\big(\bs c_q^{\bs j}, b_q^{\bs j}=f(\bs c_q^{\bs j})/N\big)\big\}}_{\text{Sample
$f$ at centers}}\;\;{\longrightarrow} \;\; 
 \underbrace{\big\{\big( a_q^{\bs j}=\Psi_q(\bs c_q^{\bs j};\bs \lambda^*), b_q^{\bs j}\big)\big\}}_{\text{Map to 1D via}\; \eqref{aqJaj}}
\;\longrightarrow \;  
    \underbrace{\big\{\big(\bs A_q, \bs B_q\big)\big\}}_{\text{Row-major ordering}},
\end{equation}
from which the outer functions are constructed by interpolating the resulting
one-dimensional data set $\{(\bs A_q,\bs B_q)\}$. 
\begin{definition}[{\bf  Approximate KST}]\label{Defn:Outer01} {\em  For $d\ge 2$ and   $0<h<1,$
let $\nu_*\in (0,h)$ and $\bs \lambda^*=(\lambda_1^*, \ldots, \lambda_d^*)$  be  given in Theorem \ref{Hypercube-Sep}, and let $\{(\tau_k, s_k)\}_{k=1}^{K}$ be the elements of 
$(\bs A_q, \bs B_q)$ with $K$ being the length of $\bs A_q.$ Then we define the outer function at each level $1\le q\le N$ as 
\begin{equation}\label{qj2g01-A}
  g_q(t)=s_k +(s_{k+1}-s_k) S\big(\tfrac{t-\tau_k}{\tau_{k+1}-\tau_k}\big), \quad t\in  [\tau_k, \tau_{k+1}],
\end{equation}
for $k=1,\ldots, K-1,$ where $S$ is the smeared-out Heaviside function  defined in \eqref{St}. Accordingly, the approximate Kolmogorov  representation of $f\in C([0,1]^d)$  takes the form 
\begin{equation}\label{0fx-KA-form-B00}
      f_N(\bs x)=f_N(\bs x; \bs \lambda^*)=\sum_{q=1}^{N} g_q\circ\Psi_q(\bs x;\bs \lambda^*)\;\;\;
      {\rm with}
      \;\;\; \Psi_q(\bs x;\bs \lambda^*)=\sum_{p=1}^d \lambda_p^*\, \psi_q(x_p),
\end{equation}
where the inner functions $\psi_q$ are defined in Definition \ref{Inner-fun}.
}
\end{definition}

\begin{figure}[!h]
\centering  
 \subfigure[$f(x_1,x_2)$ in \eqref{egzero}]{
    \includegraphics[width=0.27\textwidth]{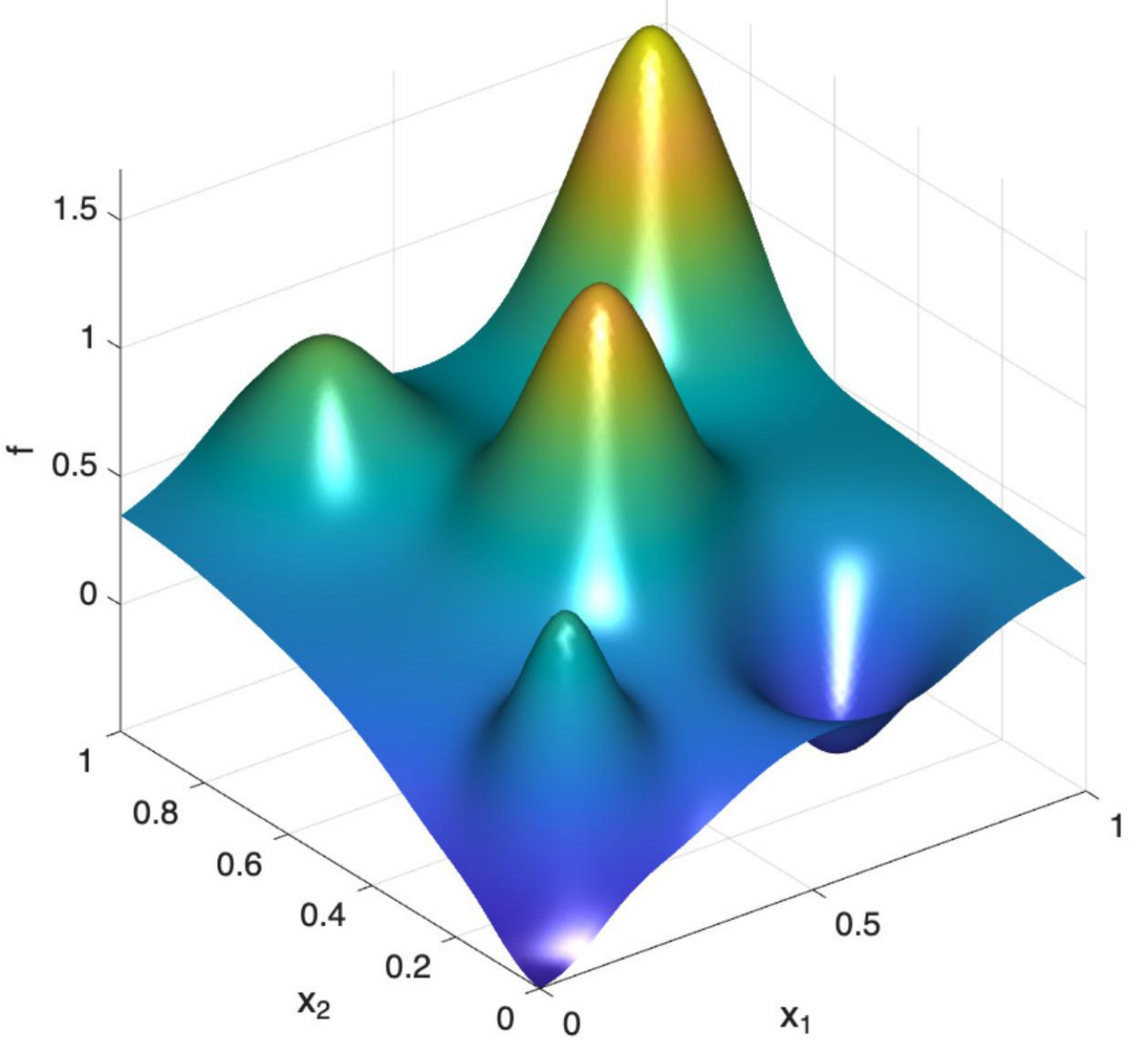}
     }\qquad 
      \subfigure[Sampling $f$  via \eqref{aqJaj2-A0}]{
    \includegraphics[width=0.25\textwidth]{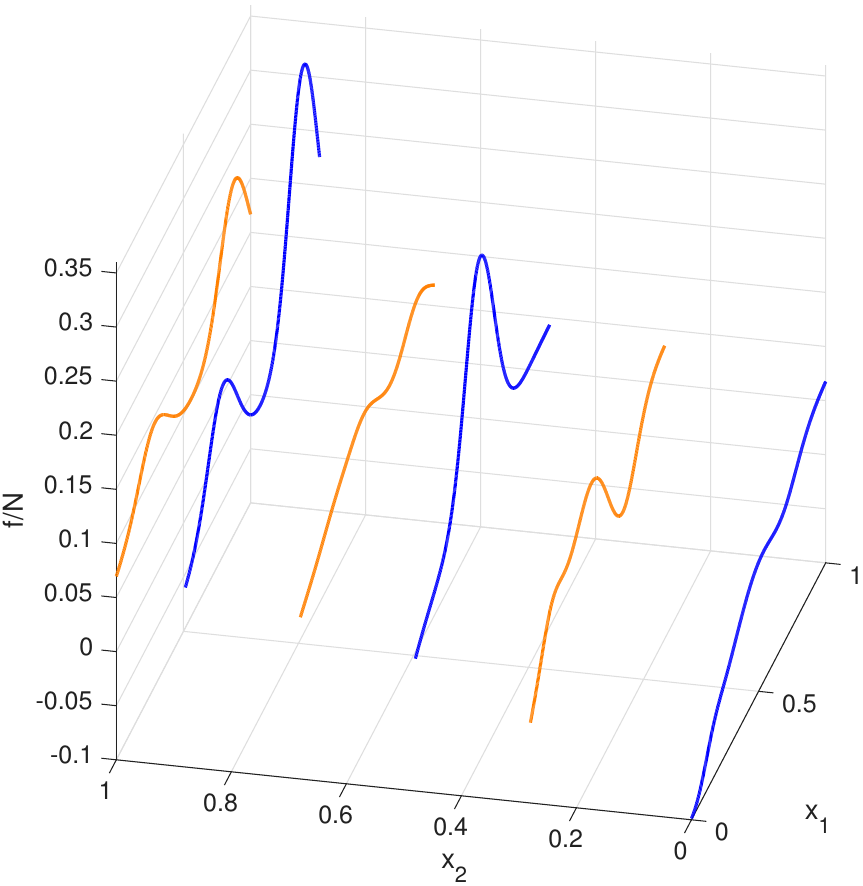}
     }\qquad
        \subfigure[$S$-interpolation]{
    \includegraphics[width=0.25\textwidth]{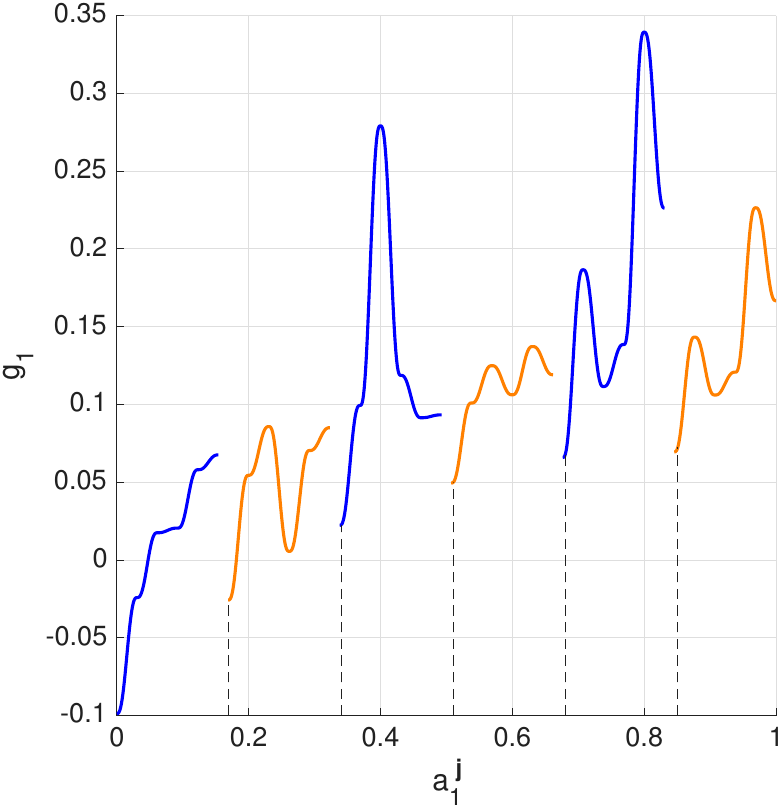}
    }
    \caption{(a) Surface plot of $f(x_1,x_2)$ in \eqref{egzero}. (b) Plots of the curves $f(:, c_1^{j_2})/N$
    with $c_1^{j_2}$ being the $x_2$-coordinate of the  centers as in Figure \ref{figsub2D-A0}~(a). (c) Row-wise interpolation as in \eqref{qj2g01-A} with $q=1$. Note that  
    $g_1(t)$ in \eqref{qj2g01-A} also connects the jumps between the six curves. 
     }
    \label{figC2function}
\end{figure}

\begin{rem}\label{S-g-remakr}
\wll{\em Using the properties of $S$, we know that $g_q:[0,1]\to\mathbb R$ defined in \eqref{qj2g01-A} is a piecewise $C^2$ function. In fact, one may also use the shape function $P$ for interpolation, and possibly interpolate every several points of $\boldsymbol A_q$.  \qed}
\end{rem}

To illustrate the data flow and row-wise interpolation, we consider  a two-dimensional example  (see Figure \ref{figC2function}~(a)):
   \begin{equation}\label{egzero}
\begin{split}
    f(x_1, x_2)&=\tanh(\ln|\sin(x_1+x_2)+\epsilon|)+e^{x_1\tan x_2})/2\\
   &\quad\;  +\sum\limits_{i=1}^5  \alpha_i\, e^{-(x_1 - 
   \beta_{1,i})^2/\sigma_i^2 - (x_2 - \beta_{2,i})^2/\sigma_i^2},
   \end{split}
\end{equation}
with $\bs \alpha=(0.75, 0.5, 1, -0.6, 1.2),$
$\bs \beta_1=(0.16, 0.28, 0.48, 0.68, 0.88),
    \bs \beta_2=(0.15, 0.88, 0.48, 0.18,$  $0.88),  
    \bs\sigma=(0.1, 0.15, 0.15, 0.15, 0.15),$ and $\epsilon=10^{-8}.$  
We slice and sample the surface \( f \) along  \( x_2 \)-coordinate of the centers shown in Figure~\ref{figC2function} (b). 
The curves in
Figure~\ref{figC2function} (c) are obtained from row-wise interpolation of the sampled data at the centers from (b).

\section{Main approximation result and numerical verifications}
\label{Sect:MainProof-1}

In this section, we present the main convergence result for the approximate Kolmogorov superpositions constructed in Definition \ref{Defn:Outer01}. We provide a detailed proof and numerical results to verify the predicted convergence order.
\subsection{Main result}
\begin{theorem}\label{MainResult} Let $d\ge 2$ and 
assume that  $f\in {\mathcal H}^\alpha([0,1]^d)$ is any $\alpha$-H\"older continuous function satisfying 
$f\in C([0,1]^d)$ and 
\begin{equation}\label{alpha-cond}
    |f(\bs x)-f(\bs y)|\le c_{\alpha}\, \|\bs x-\bs y\|_2^\alpha, \quad \forall\, \bs x,\bs y \in [0,1]^d,\;\;\; 0<\alpha\le 1.
\end{equation}
 Let $f_N(\bs x)$ be the constructed superpositions as in Definition \ref{Defn:Outer01}. Then we have 
\begin{equation}\label{approximA}
\|f-f_N\|_\infty \le  \Big(1+\frac{d^{\frac{\alpha}{2}}}{2^\alpha} +\frac{2 d^{\frac \alpha 2+1 }}{N^{1-\alpha}}\Big)\frac {c_{\alpha}} {N^{\alpha}}.
\end{equation}
Here, $\|\bs z\|_2$ is the length of the vector $\bs z$, and $\|\cdot\|_\infty$ is the usual $L^\infty$-norm.  
\end{theorem}
\begin{proof} 
By definition,  each $\bs x\in [0,1]^d$ is covered by a hypercube   
$\bs C_q^{\bs j}$ for  at least $(N-d)$ different values of $q\in\{1,\ldots, N\}$, but not covered by at most $d$ values of $q$ (i.e., falls into a gap). 
Then we can split the error into 
\begin{equation}\label{ffNas}
\begin{aligned}
|f(\bs x)-f_N(\bs x)|& = \Big|f(\bs{x})-\sum_{q=1}^N g_q\circ\Psi_q(\bs{x})\Big|\\
& =
\Big|\sum_{q=1}^N \Big(\frac {f(\bs{x})} N- g_q\circ\Psi_q(\bs{x})\Big)\Big| \leq \sum_{q=1}^N\Big|\frac{f(\bs x)}{N}-g_q\circ\Psi_q(\bs x)\Big| \\
& =  \sum_{q\in {\mathcal Q}}\Big|\frac{f(\bs x)}{N}-g_q\circ\Psi_q(\bs x)\Big|+\sum_{q\not\in {\mathcal Q}}\Big|\frac{f(\bs x)}{N}-g_q\circ\Psi_q(\bs x)\Big|,
\end{aligned}
\end{equation}
where  ${\mathcal Q}$ denotes the set of at least $N-d$ values of $q$ for which the hypercubes $\bs C_q^{\bs{j}}$ cover the point $\bs x,$ and the second summation over $q\not\in {\mathcal Q}$ counts the contributions of the levels where $\bs x$ is in a gap.  

If $q\in {\mathcal Q},$ i.e.,  $\bs x\in  \bs C_q^{\bs{j}}$  for some $\bs j$ with the center $\bs{c}_{q}^{\bs{j}}$ (see \eqref{cqjcenter}),  it follows from \eqref{alpha-cond}  that
\begin{equation}\label{Oscf2}
|f(\bs{x})-f(\bs{c}_q^{\bs{j}})|\leq c_{\alpha}\,\|\bs{x}-\bs{c}_q^{\bs{j}}\|^{\alpha}_2 \leq c_{\alpha} \Big(\frac{\sqrt{d}(N-1)\delta}{2}\Big)^{\alpha}.
\end{equation}
Correspondingly, we can identify from  \eqref{mon-hyperb} that the image $t=\Psi_q(\bs{x};\bs \lambda^*)\in [e_q^{\bs j}, \tilde e_q^{\bs j}]$ must lie in a piece, say $\tau_k, \tau_{k+1}$ in \eqref{qj2g01-A}.  Thus,   
 we  derive from   \eqref{Oscf2} that 
\begin{equation}\label{Nsquare2}
\begin{split}
\Big|\frac{f(\bs{x})}{N}  -g_q\circ\Psi_q(\bs{x})\Big| & \le  \Big|\frac {f(\bs{x})}N - \frac   {f(\bs{c}_q^{\bs{j}})} N\Big|+\Big|\frac   {f(\bs{c}_q^{\bs{j}})} N-g_q\circ\Psi_q(\bs{x};\bs \lambda^*)\Big| \le \frac 1 N \big|{f(\bs{x})}-   {f(\bs{c}_q^{\bs{j}})} \big|\\[4pt]
&\quad +\big|(s_{k+1}-s_k) S\big((\Psi_q(\bs x;\bs \lambda^*)-\tau_k)/(\tau_{k+1}-\tau_k)\big)\big|\\[4pt]
&\le \frac 1 N \big|{f(\bs{x})}-   {f(\bs{c}_q^{\bs{j}})} \big|+|s_{k+1}-s_k | \\[4pt]
&\le  \frac{c_{\alpha}} N \Big(\frac{\sqrt{d}(N-1)\delta}{2}\Big)^{\alpha}+ \frac{c_{\alpha}}{N} (N\delta)^{\alpha}
\\[4pt]
&\leq \frac{c_{\alpha}}{N}  \Big(1+\frac{d^{\frac{\alpha}{2}}} {2^{\alpha}}\Big)(N\delta)^{\alpha},
\end{split}
\end{equation}
where we have used $S(z)\in [0,1]$ (noted: $|\bs \lambda^*|_1=1$) and  \eqref{alpha-cond} with the function values in the corresponding row:  
$$s_{k+i}=b_{q,k+i}^{\bs j}=\tfrac{f(\bs{c}_{q,k+i}^{\bs{j}})}{N},\quad i=1,2,$$  to bound
\begin{equation}\label{bqjest}
\begin{split}
|s_{k+1}-s_k | &=|b_{q,k+1}^{\bs j,j_p}-b_{q,k}^{\bs j,j_p}|=\frac 1 N\big|f(\bs c_{q,k+1}^{\bs j,j_p})-f(\bs c_{q,k}^{\bs j,j_p})\big|\\[4pt]
  &\le \frac {c_{\alpha}} N \big\|\bs c_{q,k+1}^{\bs j,j_p}-\bs c_{q,k}^{\bs j,j_p}\big\|_2^\alpha  \leq \frac{c_{\alpha}}{N} (N\delta)^{\alpha}.
  \end{split}
\end{equation}
Then we obtain estimate of the first term in \eqref{ffNas}:
\begin{equation}\label{term1-44}
    \sum_{q\in {\mathcal Q}}\Big|\frac{f(\bs x)}{N}-g_q\circ\Psi_q(\bs x)\Big| \le \frac{c_{\alpha}(N-d)}{N}  \Big(1+\frac{d^{\frac{\alpha}{2}}} {2^{\alpha}}\Big)(N\delta)^{\alpha}.
\end{equation}
It remains to consider  $q\notin {\mathcal Q},$ i.e.,
$\bs x$ is in a gap. Again letting  
$t=\Psi_q(\bs{x};\bs \lambda^*),$ 
we can identify the corresponding piece $t\in [\tau_\ell,\,  \tau_{\ell+1}],$ associated with the centers and data:
$$
\tau_\ell= \Psi_q(\bs{c}_q^{\bs i};\bs \lambda^*),\quad  \tau_{\ell+1}= \Psi_q(\bs{c}_q^{\bs i'};\bs \lambda^*), \quad  s_\ell= f(\bs{c}_q^{\bs i})/N,\quad  s_{\ell+1}= f(\bs{c}_q^{\bs i'})/N.
$$
We choose the center closest to $\bs x,$ say $\bs{c}_q^{\bs i}.$ Following the first three lines of \eqref{Nsquare2}, we have 
\begin{equation}\label{Nsquare2-00}
\begin{split}
\Big|\frac{f(\bs{x})}{N}  -g_q\circ\Psi_q(\bs{x})\Big| & 
\le \frac 1 N \big|{f(\bs{x})}-   {f(\bs{c}_q^{\bs{i}})} \big|+|s_{\ell+1}-s_\ell|\\[4pt] 
& =\frac 1 N \big|{f(\bs{x})}-   {f(\bs{c}_q^{\bs{i}})} \big|+\frac 1 N \big|f(\bs{c}_q^{\bs i})-   f(\bs{c}_q^{\bs{i'}}) \big|\\[4pt]
&\le \frac{c_{\alpha}} N \big(\|\bs x- 
\bs{c}_q^{\bs{i}}\|_2^\alpha +
\|\bs{c}_q^{\bs{i}}-\bs{c}_q^{\bs{i'}}\|_2^\alpha\big).
\end{split}
\end{equation}
Note that the distance between $\bs{c}_q^{\bs{i}}$ and $\bs{c}_q^{\bs{i'}}$ can be controlled by the longest distance between two adjacent $\bs v_q^{\bs j}$ and 
$\tilde {\bs v}_q^{\bs j}$ in 
\eqref{lower-corner1}. From \eqref{secondcondforanydd}, we find  
\begin{equation}\label{cqi-est}
\|\bs{c}_q^{\bs{i}}-\bs{c}_q^{\bs{i'}}\|_2
\le \| \bs v_q^{(-1,-1,\ldots,-1,j_d+1)}-\tilde {\bs v}_q^{(N-1,N-1,\ldots,N-1,j_d)} \|_2
\le \sqrt{d-1+\delta^2}.
\end{equation}
The choice of $\lambda^*$ can ensure the ordering of images of the points in the hypercubes in rows, but cannot guarantee the points in the gaps. In other words, although $t\in [\tau_\ell,\,  \tau_{\ell+1}],$  its pre-image $\bs x$ may be far away from the pre-images $\bs{c}_q^{\bs{i}}$ and $\bs{c}_q^{\bs{i'}}$ (e.g, for the gaps of the largest volume  $(h-\delta)^{d-1}\delta$), the worst case can always be bounded by 
\begin{equation}\label{cqi-est2}
\|\bs x- 
\bs{c}_q^{\bs{i}}\|_2\le \sqrt d.
\end{equation}
Thus from  \eqref{Nsquare2-00} and the above,  we can obtain  
\begin{equation}\label{term1-44-00}
    \sum_{q\not \in {\mathcal Q}}\Big|\frac{f(\bs x)}{N}-g_q\circ\Psi_q(\bs x)\Big| \le  \frac{2c_{\alpha}}{N} d^{\frac \alpha 2+1}.  
\end{equation}
Then the estimate \eqref{approximA} is a direct consequence of  \eqref{ffNas}
\eqref{alpha-cond}, \eqref{term1-44} and \eqref{term1-44-00}. 
\end{proof}

\wll{Some remarks and discussions on the above convergence result are in order. 
\begin{itemize}
\item[(i)]  With the choice of $\nu^*$ and 
$\bs \lambda^*$ in Theorem \ref{Hypercube-Sep}, we have the  uniform convergence order $O(N^{-\alpha})$ 
with an explicit dependence on $d.$ We find from \eqref{term1-44-00} that the order from the points in the gaps is $O(N^{-1}),$  which is induced by the possible dislocations of the data as shown in \eqref{cqi-est}-\eqref{cqi-est2}, and independent of the smoothness of $f$. With 
controlling the hypercubes via $\bs E_q$ in Lemma \ref{lmm:Hypercube-Sep} and Theorem \ref{Hypercube-Sep},   the convergence order is adapted to the smoothness.   Overall, even for sufficiently smooth $f$, the highest expected order should be $O(N^{-1}).$     
\smallskip
\item[(ii)] Observe that the parameter $\nu^*=h-h^d+O(h^{d+1})$; see Corollary \ref{lamda-small}. Although it does not appear explicitly in the error bound, it plays an important role in the whole construction. Indeed, from \eqref{shapfunc0}, we see that $h-\nu^*=h^d+O(h^{d+1})$, which controls the deviation of $\phi(t)$ and $\psi_q(x)$ from constant values. In particular, the higher the dimension is, the closer the construction is to constant pieces on the hypercubes. In fact, it is clear from our explicit construction that this requirement results from the separation and ordering of the hypercubes. Therefore, the higher the dimension is, the stricter the requirement becomes.  {\em To our knowledge, this insight has not been revealed in any existing constructions.} 
\smallskip
\item[(iii)] Although our construction is designed for non-constant pieces on the hypercubes, it is interesting to examine the degenerate case $h=\nu$ in \eqref{shapfunc0}.  
In this case, the equalities in \eqref{lambdapsi}--\eqref{mon-hyperb} hold, and hence $e^{\bs j}_q=a^{\bs j}_q=\tilde e^{\bs j}_q$. In other words, all points in a hypercube are mapped to a single value, so any point in the hypercube can be chosen as an abscissa for interpolation, say $a_q^{\bs j}$. Accordingly, the analysis of $\bs E_q$ reduces to that of $\bs A_q$, for which the separation property \eqref{distanceajj-00} holds. Moreover, an analysis of the heads and tails of the rows, similar to that in Lemma \ref{lmm:Hypercube-Sep}, can be carried out to derive the conditions on $\bs\lambda$. A similar convergence estimate should still hold. However, the resulting inner functions have vanishing derivatives in all hypercubes. We note that the constructive proof in \cite{Kahane1975} and the constructions in \cite{Igelnik2003} adopt such constant structures.
\end{itemize}
}

\subsection{Numerical algorithm and verifications} 
We first summarize and present the computation of the approximate Kolmogorov's superpositions in 
Algorithm \ref{alg:empirical_Linf_fN}. To alleviate 
the storage requirement in high dimensions, we parallelize the implementation so that each sample $\bs x$ and each layer $q\in \{1,\ldots, N\}$ (see Lines 1-14 in 
Algorithm \ref{alg:empirical_Linf_fN}) can be computed in parallel.   It follows from the proof of Theorem \ref{MainResult} that, when $\bs x$ falls into a gap,
which can occur at most $d$ levels among the total $N$ levels, the map
$\Psi_q$ is non-injective in the sense that any $t=\Psi_q(\bs x;\bs \lambda^*)\in [0,1]$ may
correspond to many points $\bs x\in [0,1]^d.$ For simplicity, we associate it with the nearest hypercube; see Lines 4–6. In fact, the analysis in \eqref{Nsquare2-00}–\eqref{cqi-est} shows that it can also be handled in other ways.  

\begin{algorithm}[htbp]
\caption{\bf Algorithm for computing $f_N(\bs x)$ and $e_N(\bs x)=|f(\bs x)-f_N(\bs x)|$}
\label{alg:empirical_Linf_fN}
\begin{algorithmic}[1]
\Require  Function:
$f(\bs x)\in {\mathcal H}^\alpha([0,1]^d)$ 
at any randomly sampled  point $\bs x=(x_1,\dots,x_d)\in[0,1]^d$;  
Parameters:
$N,~h=1/N,~\delta=1/N^2,$ and $\nu^\ast,\,
\boldsymbol{\lambda}^\ast=(\lambda_1^\ast,\lambda_2^\ast,\dots,\lambda_d^\ast)
$ given by
\eqref{qAeqnA} and \eqref{def_d_lambda} 
\smallskip
\Ensure The value
   $f_N(\bs x)$ and the point-wise error $e_N(\bs x)=|f(\bs x)-f_N(\bs x)|$
\smallskip

\For{$q=1,2,\dots,N$} \textbf{in parallel}  \smallskip
    \State Set $\boldsymbol{j}=(j_1,\dots,j_d)$ with $j_p=\big\lfloor \frac{x_p-q\delta}{h}\big\rfloor,$ and compute  $\bs{v}_q^{\bs j},~\tilde{\bs{v}}_q^{\bs j},~\bs{c}_q^{\bs j}$ by \eqref{lower-corner1}-\eqref{cqjc-00}
    \smallskip
    \For{$p=1,2,\dots,d$} 
        \If{$x_p>\tilde v_q^{j_p}$}
            \State $x_p\leftarrow x_p-\delta$
        \EndIf
    \EndFor
    \smallskip
        \State Compute $t=\Psi_q(\bs x)=\sum_{p=1}^d \lambda_p^\ast\,\psi_p(x_p)$ by \eqref{d-PsiA0}
        \smallskip
        \If{$\Psi_q(\bs x)\le\Psi_q(\bs{c}_q^{\bs j})$} 
        \smallskip
            \State Evaluate $g_q(t)$ with $\tau_k=\Psi_q(\bs{c}_q^{\bs j,\mathrm{-}})$ and $\tau_{k+1}=\Psi_q(\bs{c}_q^{\bs j})$
        by \eqref{qj2g01-A}, where 
        \hspace*{1cm} $\Psi_q(\bs{c}_q^{\bs j,\mathrm{-}})$ denotes the element to closest  $\Psi_q(\bs{c}_q^{\bs j})$ in  $\bs A_q$ \eqref{ufform}, lying to its left
        \Else
        \smallskip
            \State Evaluate $g_q(t)$ with $\tau_{k}=\Psi_q(\bs{c}_q^{\bs j})$ and $\tau_{k+1}=\Psi_q(\bs{c}_q^{\bs j,\mathrm{+}})$
        by \eqref{qj2g01-A}, where \hspace*{1cm} $\Psi_q(\bs{c}_q^{\bs j,\mathrm{+}})$ is the  element closest to $\Psi_q(\bs{c}_q^{\bs j})$ in $\bs A_q$, lying to its right     
        \EndIf
        \smallskip
\EndFor
\smallskip
\State Collect $\{g_1\circ \Psi_1(\bs x), \ldots, g_N\circ \Psi_N(\bs x)\},$  and 
compute $f_N(\bs x)=\sum_{q=1}^Ng_q\circ\Psi_q(\bs x)$
\smallskip
\State Compute the pointwise error $e_N(\bs x)=|f(\bs x)-f_N(\bs x)|$
\smallskip
\State 
\smallskip 
\Return $f_N(\bs x)$ and $e_N(\bs x)$
\end{algorithmic}
\end{algorithm}

In the first example, we consider 
$f(\bs x)=u(\bs x)\ln |u(\bs x)|,$ where $u(\bs x)$ is  given by 
\begin{equation}\label{eg2D3D4D}
\begin{split}
  u(\bs x)&=
\tanh(\ln(|\sin(x_1+\cdots+x_d)|+\epsilon)+e^{x_1\cos (x_2+\cdots+x_d)})/2\\
   &\quad\;  +\sum\limits_{i=1}^5  \alpha_i\, e^{-(x_1 - 
   \beta_{1,i})^2/\sigma_i^2 - (x_2 - \beta_{2,i})^2/\sigma_i^2\cdots-(x_d-\beta_{d,i})^2/\sigma_i^2},
   \end{split}
\end{equation}
where we choose the constants so that $|u(\bs x)|$ can take zero values at various locations. We set  $f=0$ if $u=0.$ Here we  take  $\bs \alpha=(0.75, 0.5, 1, -0.6, 1.2),$
$\bs \beta_1=(0.16, 0.28, 0.48, 0.68, 0.88),
    \bs \beta_2=(0.15, 0.88, 0.48, 0.18,$  $0.88),  \bs \beta_3=(0.06, 0.15, 0.37, 0.59, 0.91), $ 
    $ \bs \beta_4=(0.04, 0.27, 0.39, 0.53, 0.85),    $ $\bs \beta_5=(0.01, 0.11, 0.26, 0.67, 0.77), $ 
$\bs \beta_6=(0.12, 0.22, 0.46, 0.87, 0.97),$ $
    \bs \beta_7=(0.25, 0.48, 0.55, 0.67,$ $ 0.7),
    $ $\bs \beta_8=(0.13, 0.41, 0.56, 0.77, 0.87),
    $ $\bs \beta_9=(0.02, 0.32, 0.36, 0.57, 0.83),
    $ $\bs\sigma=(0.1, 0.15, 0.15,$ $ 0.15, 0.15),$ and $\epsilon=10^{-8}.$  
    It is noteworthy that $f$ is $\alpha$-H\"older continuous with any $0<\alpha<1$ (see e.g., \cite{Wang2024IMEX}).

 \wll{In the test, we randomly sample $M$ points
 $\mathcal X_M$ for $2\le d\le 9,$ and then calculate the maximum point-wise error 
 $\max\limits_{\bs x\in {\mathcal X}_M} e_N(\bs x).$ Here,  we take 
 $M=1e4, 2.7e4, 1.6e4, 1e5$ for $d=2,3,4,5,$ respectively, and plot in Figure \ref{fig:ulogu-A0}~(a) the errors  against $N=32,64,128,256,512,1024$ in log-log scale, which indicates roughly a first-order convergence as predicted in Theorem \ref{MainResult}.  In higher dimension, we take
 $M=6^6,5^7,4^8,4^9=46656,  78125, 65536, 262144$ for  $d=6,7,8,9,$ respectively. In fact, as $d$ increases, the storage becomes the major concern, and the speed is much less concerned.  
 For example, for $d=5$, the algorithm for  $N=32,64,128,256,512,1024$ with the above samples  
 takes only about $9$ minutes on a standard laptop. The convergence order illustrated in Figure \ref{fig:ulogu-A0}~(b) again well agrees with the predicted rate  
 in Theorem \ref{MainResult}.}

\begin{figure}[!h]
   \centering
        \subfigure[$d=2,3,4,5$]{
\includegraphics[width=0.45\textwidth]{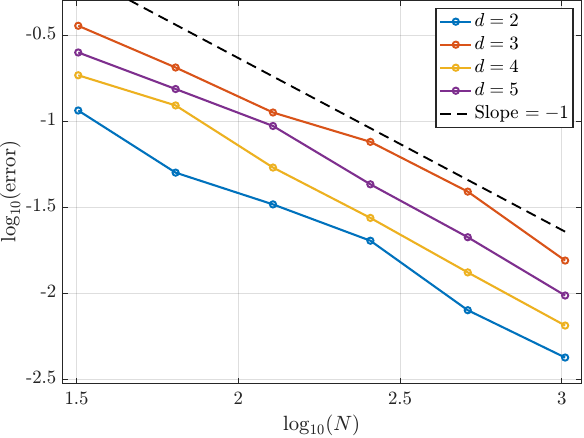}
     } \qquad 
        \subfigure[$d=6,7,8,9$]{
\includegraphics[width=0.45\textwidth]{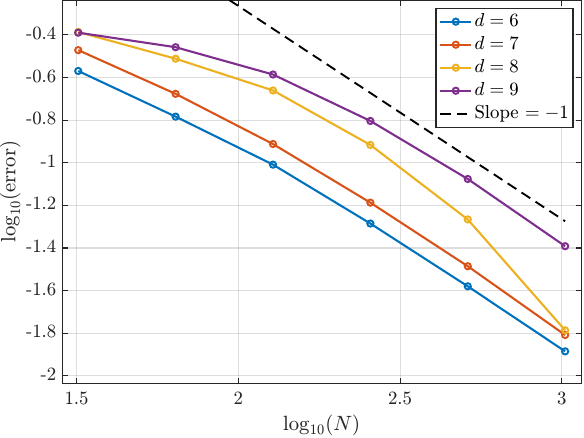}
    } 
    \caption{Convergence rates of the approximate superpositions for the function defined in \eqref{eg2D3D4D} for $N=32,64,128,256,512,1024$.  Left:  $d=2,3,4,5.$ Right: $d=6,7,8,9$.}
 \label{fig:ulogu-A0}
\end{figure}

\wll{In the second example, we test the approximate on 
\begin{equation}\label{eg2D3D4D_2}
f(\boldsymbol{x})
=
\sin\Big(
\sum_{p=1}^{d}
|
x_p
-x_p^0|^{\alpha}\Big),\quad x_p^0:= 0.31+\frac{0.42(p-1)}{d-1},
\end{equation}
for $\alpha\in (0,1),$ which is singular at $x_p^0$ and belongs to ${\mathcal H}^\alpha([0,1]^d).$
In addition to the sampling points as in the first  example, we properly sample more points near singularities at $x_p=x_p^0,$ to compute the errors more accurately.  In Figure \ref{fig:ulogu2-A1}, we plot the convergence rates for   \eqref{eg2D3D4D_2} with $\alpha=0.25$ in the same setting as the first example.  Remarkably, the constructed approximate superpositions converge at the same rate as predicted. 
\begin{figure}[!h]
   \centering
        \subfigure[$d=2,3,4,5$]{
\includegraphics[width=0.45\textwidth]{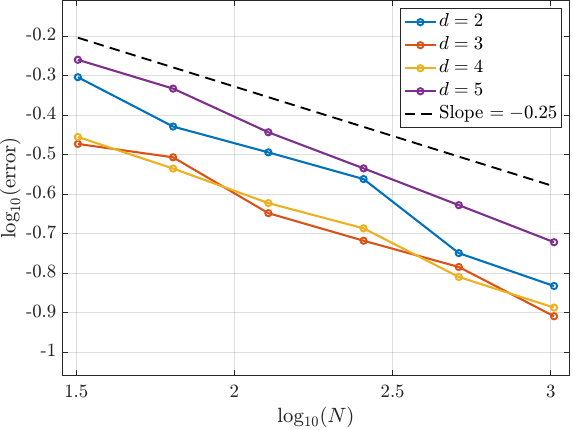}
     } \qquad 
        \subfigure[$d=6,7,8,9$]{
\includegraphics[width=0.45\textwidth]{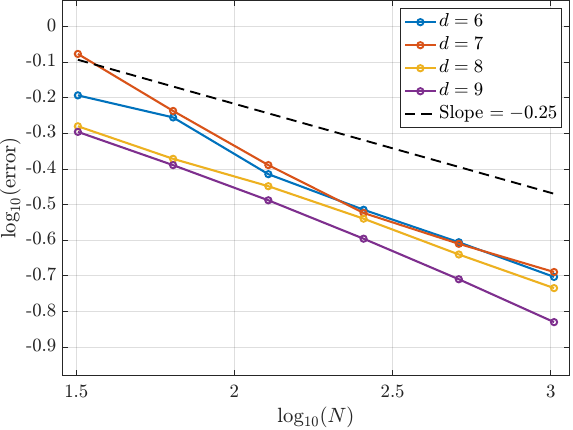}
    } 
    \caption{Convergence rates of the approximate superpositions for the function defined in \eqref{eg2D3D4D_2} with $\alpha=0.25$ for $N=32,64,128,256,512,1024$.  Left:  $d=2,3,4,5.$ Right: $d=6,7,8,9$.}
 \label{fig:ulogu2-A1}
\end{figure}

 From the above figures, the dependence on $d$ does not appear to be clearly observable. This may be because the H\"older continuity constant $c_\alpha$ also depends on $d$. 
}

\subsection{Concluding remarks}

It is well known that the univariate functions involved in the exact Kolmogorov superpositions lack differentiability and smoothness, often exhibiting “wild” behavior—even when the target function 
$f(\bs x)$ is smooth. This significantly hinders the practical application of this mathematically elegant theory.  It was suggested that one should sacrifice the exact representation and use an approximate one instead \cite{Kourkova1991,Kuurkova1992kolmogorov}. However, the explicit construction of a smooth approximation version is still open (see \cite{Demb2021note} and \cite[Sec.~2.4]{Shen2021three}).  \wll{This work resolves this open issue by explicitly constructing an approximate $C^2$ version that is more practically implementable, together with rigorous analysis and supporting numerical verifications.}


 We emphasize that our construction relies on  the new shape functions defined in \eqref{St}--\eqref{P-fun}, which possess desirable properties enabling the construction of $C^2$, strictly increasing inner functions. Moreover, the  Kolmogorov maps defined by the superpositions of resulting inner functions allow us to explicitly study the separation of multi-dimensional grids mapped to one dimension, as well as various other properties that have not been quantitatively investigated in existing constructions. It also informs the row-wise interpolation strategy used in the explicit construction of the outer functions, which, to the best of our knowledge, has not been explored in the existing literature.

 As aforementioned, one implication of this approximate version is to serve as an intermediate tool to construct networks and study their 
convergence.  For example,  we consider the general  approximators (see \cite{Toscano2024kkans}) 
\begin{equation}\label{KappM}
f_N^{\Theta}(\boldsymbol{x})=\sum_{q=1}^{N}  \tilde g_q\circ\sum_{p=1}^d \lambda_p \tilde \psi_{p}(x_p),\quad  \forall\, \tilde g_q\in {\mathcal A}_{M_g}(I),\; \forall\, \tilde \psi_{p}\in   {\mathcal A}_{M_\psi}(I), 
\end{equation}
where the ansatz spaces ${\mathcal A}_{M_g}(I)$ and ${\mathcal A}_{M_\psi}(I)$ with $I:=[0,1]$ 
 are chosen to be dense subsets of the spaces of continuous functions $C(I).$   Here, $\Theta$ denotes the set of all involved parameters in the approximators.  This two-block structure followed  the spirit of K{$\dot {\rm u}$}rkov{\'a} \cite{Kourkova1991,Kuurkova1992kolmogorov} and inspired the development of KKANs in \cite{Toscano2024kkans}, which enjoy universal approximability and exhibit rich learning dynamics. 
However, as the one-dimensional functions in the KST involving exact superpositions 
 lack differentiability and smoothness,   it is still open to prove the convergence of KKANs.   With this  approximate version at our disposal, we can formally write
 \begin{equation}\label{newcons}
 \|f-f_N^{\Theta}\|_\infty\le   \|f-f_N\|_\infty +  \|f_N-f_N^{\Theta}\|_\infty,
 \end{equation}
 so leveraging existing approximation results for the second term, our approach paves the way toward a convergence theory for KKANs. We will report the development along this line in future work. 

 We also point out that Kolmogorov's representations have a profound influence on both theoretical development and practical implementation of neural networks.
 For instance, the neural network constructions in \cite{Shen2021three,Lu2021deep,Shen2022optimal} rely on $\delta$-shifting architectures (see,  e.g., Figure~\ref{figsubLq2D}~(a)), while the analytical framework in the recent work \cite{He2024} is grounded in Kolmogorov--Arnold representation theory. In contrast, our construction takes a fundamentally different route, offering fresh perspectives on the theoretical underpinnings of deep neural networks. In particular, it shows some potential for establishing a theoretical foundation for the recently proposed  Kolmogorov--Arnold networks (KANs)  \cite{Liu2024KAN}. In fact, the convergence theory of KANs is based on assumptions about the existence of multi-layer superpositions of the target function \(f(\bs x)\) and the smoothness of the one-dimensional functions involved in these superpositions. In light of the superposition results in \cite{Demko1977superposition,Doss1977superposition}, our approach appears promising for putting the KAN theory on a solid footing.

\wll{As a final remark, the construction of the outer functions in this work is based on global row-wise interpolation of the mapped hypercube centers under the Kolmogorov maps that define the superpositions of the explicit inner functions. This global nature (in the spirit of Kolmogorov's representations)  requires the parameters to be selected as specified in Theorem~\ref{Hypercube-Sep}. In particular, the choice $\nu=\nu^*$ makes the inner functions become nearly flat within the hypercubes as the dimension $d$ increases, so the derivatives of the inner functions become nearly zero in hypercubes with large variations in gaps.
This may be viewed as a natural consequence of encoding $d$-dimensional information through one-dimensional superpositions. However, this restriction may be significantly relaxed if one reduces the dimension only from $d$ to $d-1$ or, more generally, to a lower dimension, rather than all the way to one dimension, as in the earlier pre-1957 superpositions of Kolmogorov and Arnold. We shall explore this direction in future work. }

\bigskip 
\noindent{\bf Acknowledgments:}\, The authors gratefully acknowledge Professor George Em Karniadakis of Brown University for many insightful discussions on this topic and, in particular, for his valuable suggestion to include the keywords ``explicit construction'' in the title.  The authors would also like to thank Professor Zhongjian Wang of Nanyang Technological University for his constructive challenges and discussions in the development of this construction.


\end{document}